\documentclass[12pt]{amsart}
\usepackage{amssymb}
\usepackage{amsmath}
\usepackage{bbm}
\usepackage{amscd}
\usepackage{dsfont}

\usepackage[dvips]{graphicx}

\usepackage{book tabs}
\usepackage{float}

\graphicspath{{Bilder/}}
\usepackage{caption}
\usepackage{subcaption}

\usepackage{enumitem}

\numberwithin{equation}{section}
\newtheorem{thm}{Theorem}
\newtheorem{prop}[thm]{Proposition}
\newtheorem{lemma}[thm]{Lemma}
\newtheorem{cor}[thm]{Corollary}

\theoremstyle{definition}

\newtheorem{example}[thm]{Example}
\newtheorem{remark1}[thm]{Remark}
\newtheorem{openproblem1}[thm]{Open problem}

\newenvironment{ex}{\begin{example}\rm}{\end{example}}

\numberwithin{thm}{section}

\newcounter{FNC}[page]
\def\newfootnote#1{{\addtocounter{FNC}{2}$^\fnsymbol{FNC}$%
     \let\thefootnote\relax\footnotetext{$^\fnsymbol{FNC}$#1}}}

\newcommand{\N}{\mathbb{N}}
\newcommand{\Q}{\mathbb{Q}}
\newcommand{\R}{\mathbb{R}}

\newcommand{\B}{\mathbb{B}}

\newcommand{\Z}{\mathbb{Z}}

\newcommand{\T}{\mathbb{T}}

\newcommand{\VV}{\mathcal{V}}
\newcommand{\HH}{\mathcal{H}}
\newcommand{\Acal}{\mathcal{A}}
\newcommand{\Bcal}{\mathcal{B}}

\newcommand{\sym}{\mathcal{S}}
\newcommand{\psd}{\mathcal{S}_{+}}
\newcommand{\pd}{\mathcal{S}_{++}}

\newcommand{\pip}{\pi(P)}

\newcommand{\pih}{\pi\HH}
\newcommand{\pipa}{\pi(P_A)}
\newcommand{\pipb}{\pi(P_B)}
\newcommand{\pis}{\pi\sym}
\newcommand{\pisa}{\pi(S_A)}
\newcommand{\pisb}{\pi(S_B)}

\newcommand{\st}{\mathrm{s.t.}}

\newcommand{\widephi}{\widehat{\Phi}}
\newcommand{\widea}{\widehat{A}}
\newcommand{\widema}{\widehat{\mathcal{A}}}

\newcommand{\bs}{\backslash}

\newcommand{\qm}{\mathcal{M}}

\DeclareMathOperator{\SOS}{\Sigma}
\DeclareMathOperator{\linspan}{span}
\DeclareMathOperator{\diag}{diag}

\DeclareMathOperator{\tr}{tr}
\DeclareMathOperator{\cl}{cl}

\DeclareMathOperator{\kernel}{ker}

\usepackage{xcolor}

\title{Containment Problems for Projections of Polyhedra and Spectrahedra}

\author{Kai Kellner}

\address{Goethe-Universit\"at, FB 12 -- Institut f\"ur Mathematik,
Postfach 11 19 32, D--60054 Frankfurt am Main, Germany}

\email{kellner@math.uni-frankfurt.de}

\usepackage{hyperref}

\begin{document}

\begin{abstract}
Spectrahedra are affine sections of the cone of positive semidefinite matrices
which form a rich class of convex bodies that properly contains that of
polyhedra.
While the class of polyhedra is closed under linear projections, the class of
spectrahedra is not.
In this paper we investigate the problem of deciding containment of
projections of polyhedra and spectrahedra based on previous works on
containment of spectrahedra.
The main concern is to study these containment problems by formulating them as
polynomial nonnegativity problems.
This allows to state hierarchies of (sufficient) semidefinite conditions by
applying (and proving) sophisticated Positivstellens\"atze.
We also extend results on a solitary sufficient condition for containment of
spectrahedra coming from the polyhedral situation as well as connections to
the theory of (completely) positive linear maps.
\end{abstract}

\maketitle

\section{Introduction}

A containment problem is the task to decide the set-theoretic inclusion of two
given sets.
In a broader sense this includes, e.g., radii~\cite{gritzmann1992} or packing
problems~\cite{banai1999}.
For some classes of convex sets there has been strong interest in containment
problems.
This includes containment problems of polyhedra and balls~\cite{freund1985}
and containment of polyhedra~\cite{gritzmann1994}.
In recent years, containment problems for spectrahedra, which naturally
generalize the class of polyhedra, have seen great interest.
  Ben-Tal and Nemirovski started that investigation by developing
approximations of uncertain linear matrix inequalities yielding a quantitative
semidefinite criterion for the so-called \emph{matrix cube problem}, the
decision problem whether a cube is contained in a
spectrahedron~\cite{bental2002}.
  Helton, Klep, and McCullough studied the geometry of so-called \emph{free
spectrahedra} (also known as matricial relaxation of spectrahedra) including
containment problems~\cite{Helton2012,Helton2010}.
They established connections to operator theory, namely equivalence between
containment of free spectrahedra and positivity of a certain linear map.
  From that they gained a sufficient semidefinite criterion for containment of
spectrahedra which coincides with the Ben-Tal-Nemirovski criterion when  
applied to the matrix cube problem.
Recently Helton, Klep, McCullough, and Schweighofer showed that the
quantitative version of Ben-Tal-Nemirovski's criterion is the best
possible~\cite{helton2014}.

The sufficient semidefinite criterion has been reproved by Theobald, Trabandt,
and the author by a geometric approach.
They also provided exact semidefinite characterizations for containment in
several important cases~\cite{ktt1}.
In a second work the authors formulated the containment problem for
spectrahedra as a polynomial nonnegativity question~\cite{ktt2}.
Based on this formulation they studied a hierarchy of semidefinite programs
each serving as a sufficient condition for containment coming from Lasserre's
moment approach~\cite{Henrion2006} and Putinar's
Positivstellensatz~\cite{putinar1993}.
It turned out that the first step of the hierarchy is implied by the solitary
criterion coming from positive linear maps and the geometric approach,
yielding finite convergence statements in several cases.
In~\cite{kphd} the author considered a different but related hierarchy of
semidefinite programs based on the Positivstellensatz by Hol and
Scherer~\cite{Hol2006}.
As this approach relies on the geometry of the spectrahedra and the defining
linear matrix pencils it allows stronger results.
Specifically, all finite convergence results from~\cite{ktt2} can be brought
forward to this hierarchy and its first step coincides with the solitary
criterion.
In addition, using the connection to the theory of positive linear maps,
finite convergence is shown for a special family of 2-dimensional bounded
spectrahedra.

This paper is concerned with containment problems for projections of polyhedra
and spectrahedra.
Besides the natural question of extending the results for containment of
polyhedra and spectrahedra to their projections, the paper is motivated by the
growing attention the geometry of projections of polyhedra and spectrahedra
attracted in recent years.
Among others they have become relevant in many areas like polynomial
optimization~\cite{blekherman2012,gouveia2010,lasserre2010}, (real) convex
algebraic geometry~\cite{Blekherman2013,helton2007}, and extended formulations
of polytopes~\cite{fiorini2012}. 

Starting point of our considerations are the methods and results discussed
above.
More precisely, we treat possible extensions of the geometric approach,
positive linear maps and polynomial optimization to the case of projections.
The main considerations and contributions are the following.

\begin{enumerate}
  \item
  Although the class of polyhedra is closed under (linear) projections,
containment problems become more subtle.
This is reflected in the fact that the containment problem of two projected
polyhedra is co-NP-complete (Theorem~\ref{thm:complexity_h_pih}).
We formulate the problem as a bilinear nonnegativity question
(Theorem~\ref{thm:pihpih}) and study its geometry.
  \item
  As the class of spectrahedra is not closed under projections, containment
problems are even more subtle.
However under additional assumptions (which are common in semidefinite
programming) a similar formulation as in the polyhedral case is possible
(Theorem~\ref{thm:pispis}).
\end{enumerate}
  Retreating to the case where only one set is given as a projection, allows to
bring forward several results from the non-projected case.
\begin{enumerate}[resume]
  \item
  Based on the polyhedral case we deduce a sufficient semidefinite criterion
for containment of a projected spectrahedron in a spectrahedron
(Theorem~\ref{thm:inclusion-proj}).
  \item
  We establish a refinement of Hol-Scherer's Positivstellensatz based on the
geometry of the projected sets (Theorem~\ref{thm:pispossatz}).
That allows to state a hierarchy of sufficient semidefinite conditions to
decide containment.
The first step of the refined hierarchy coincides with the solitary criterion.
As a corollary we gain a Positivstellensatz for polynomials on projections of
polytopes (Proposition~\ref{prop:pihpossatz}).
  \item
  The connection between containment and the concept of positive linear maps
can be extended to this case (Theorem~\ref{thm:pispos}).
\end{enumerate}

\smallskip

The paper is structured as follows.
After introducing some relevant notation, we state some basics on projections
of polyhedra and spectrahedra; see Section~\ref{sec:prelim}.
We formulate the containment problem for projections of polyhedra as a
polynomial nonnegativity question in Section~\ref{sec:pihpih} and extend it to
the case of projections of spectrahedra in Section~\ref{sec:pispis}.
In Section~\ref{sec:pis-s} we retreat to the containment problem of a
projected spectrahedron in a spectrahedron.

\section{Preliminaries} \label{sec:prelim}

Let $\sym^k$ be the space of symmetric $k\times k$ matrices with real entries
and $\psd^k$ be the closed, convex cone of positive semidefinite $k\times k$
matrices.
For $x = (x_1,\ldots,x_d)$ denote by $\sym^k [x]$ the space of symmetric
$k\times k$ matrices with entries in the polynomial ring $\R[x]$.
For real symmetric matrices $A_0,A_1,\ldots,A_d\in\sym^k$ a linear matrix
polynomial $A(x) = A_0 + \sum_{p=1}^d x_p A_p \in\sym^k[x]$ is called a
\emph{linear (matrix) pencil}.
The positivity domain of $A(x)$ is defined as the set of points in $\R^d$ for
which $A(x)$ is positive semidefinite,
\begin{equation*}
  S_A = \left\{ x\in\R^d\ |\ A(x) \succeq 0 \right\} ,
\end{equation*}
where $A(x) \succeq 0$ denotes positive semidefiniteness.
The closed, convex, and basic closed semialgebraic set $S_A$ is called a
\emph{spectrahedron}.

Every \emph{$\HH$-polyhedron} $P_A = \{x\in\R^d\ |\ a+Ax\ge 0 \}$ has a
natural representation as a spectrahedron called the \emph{normal form} of the
polyhedron $P_A$ as a spectrahedron,
\begin{equation}\label{eq:normalform}
  P_A = \left\{ x\in\R^d\ |\ A(x) = \bigoplus_{i=1}^k a_i(x) =
  \begin{bmatrix} a_1(x) & & 0 \\ & \ddots & \\ 0 & & a_k(x) \end{bmatrix}
  \succeq 0 \right\},
\end{equation}
where $a_i(x) = (a+Ax)_i$ for $i\in [k]$. 
However the converse is not true, i.e., there exist nondiagonal pencils
describing polyhedra. Deciding whether a given spectrahedron is a polyhedron,
the so-called \emph{Polyhedrality Recognition Problem} (PRP), is
NP-hard~\cite{ramana1997c} and can be reduced to an $\HH$-in-$\sym$
containment problem~\cite{Bhardwaj2011}.
Note that the normal form used here does not coincide with the normal form
used in~\cite{ktt1,ktt2} as we do not require the constant term $a$ to be the
all-ones vector.

Following the common notation for bounded polyhedra (\emph{polytopes}), we
call a bounded spectrahedron a \emph{spectratope}.


Denote by $\pi:\ \R^{d+m}\to\R^d$ $(x,y) \mapsto x$ the linear coordinate
projection map.
By Fourier-Motzkin elimination, given an $\HH$-polyhedron 
$P = \{(x,y)\in\R^{d+m}\ |\ a+Ax+A'y\geq 0\}$, the projection of $P$ onto the
$x$ coordinates is again an $\HH$-polyhedron. 
Unfortunately, a quantifier-free $\HH$-description of $\pi(P)$ can be
exponential in the input size $(d,m,k)$, where $k$ is the number of rows in
$A$ and $A'$; see~\cite[Sections 1.2 and 1.3]{Ziegler1995} and the references
therein.

Given a linear pencil $A(x,y)\in\sym^k[x,y]$ with $x = (x_1,\ldots,x_d)$ and 
$y= (y_1,\ldots,y_m)$ for some nonnegative integer $m$, a projection of the
spectrahedron $S_A$ is its image under an affine map. By an elementary
observation, without loss of generality, we can assume that the affine
projection is a coordinate projection.

\begin{prop}[{\cite[Section 2]{Gouveia2011}}] \label{prop:coordinate}
If a set $T\subseteq\R^d$ is the image of a spectrahedron $S$ under an affine 
map, then there exists a linear pencil $A(x,y)\in\sym^k[x,y]$ with 
$x = (x_1,\ldots,x_d)$ and $y = (y_1,\ldots,y_m)$ for some nonnegative integer
$m$ such that $T$ is a coordinate projection of $ S_A \subseteq \R^{d+m} $.
Furthermore, if $T$ and $S$ have nonempty interior, then this can be assumed
for $S_A$ too.
\end{prop}

Due to the proposition, we always assume that the \emph{projection of
spectrahedron} $S_A$ as given by the linear pencil $A(x,y)\in\sym^k[x,y]$ with 
$ x = (x_1,\ldots,x_d)$ and $ y = (y_1,\ldots,y_m)$, $m\geq 0$, is the set 
\begin{equation}\label{eq:proj_spec}
  \pisa = \left\{ x\in\R^d\ |\ \exists y\in\R^m:\ A(x,y)\succeq 0 \right\} ,
\end{equation}
where $\pi:\ \R^{d+m} \rightarrow \R^m$ denotes the coordinate projection.

While projections of polyhedra are again polyhedral, this is not true for
spectrahedra (see, e.g., \cite[Section 6.3.1]{Blekherman2013}).
Moreover, whereas spectrahedra are basic closed semialgebraic sets (the
semialgebraic constraints are given by the nonnegativity condition on the
principal minors), projected spectrahedra are generally not. 
Though they are semialgebraic, they are not (basic) closed in general; see
Example~\ref{ex:pispispos}.

We state an easy observation for completeness.

\begin{lemma} \label{lem:spec_bounded}
Let $ A(x,y) \in\sym^k[x,y] $ be a linear pencil. 
\begin{enumerate}
  \item
  $ S_A \neq \emptyset \iff \pisa\neq\emptyset $.
  \item 
  If $ S_A $ is bounded, then $ \pisa $ is bounded. 
\end{enumerate}
\end{lemma}

The converse of part (2) in the previous lemma is not true in general.

Throughout the paper we use the following notation.
The class of projections of $\HH$-polyhedra (resp. spectrahedra) is denoted by
$\pih$ (resp. $\pis$).
For integers $m,n\in\Z$ with $m\leq n$ we write $[m,n] = \{m,m+1,\ldots,n\}$
and $[m]=[1,\ldots,m]$.

\subsection{Complexity of Containment Problems}\label{sec:complexity}

The computational complexity of containment problems concerning polyhedra is
well-known~\cite{freund1985,gritzmann1993,gritzmann1994,gritzmann1995}.
Recently this has been extended to spectrahedra~\cite{bental2002,ktt1}.
We shortly classify the complexity of several containment problems for
projections of $\HH$-polytopes and spectrahedra. For more details on the
complexity classification see~\cite{kphd}.

Our model of computation is the binary Turing machine: 
projections of polytopes are presented by certain rational numbers, and the
size of the input is defined as the length of the binary encoding of the input
data (see, e.g.,~\cite{gritzmann1992}). 
Consider the linear projection map $\pi:\ \R^{d+m}\to\R^d,\ (x,y)\mapsto x$.
An $\pih$-polytope $\pi(P)$ is given by a tuple $(d;m;k;A;A';a)$ with
$d,m,k\in\N$, matrices $A\in\Q^{k\times d}$ and $A'\in\Q^{k\times m}$, and
$a\in\Q^{k}$ such that 
$\pi(P) = \{x\in\R^d\ |\ \exists y\in\R^m:\ a+Ax+A'y \ge 0\}$ is bounded. 
For algorithmic questions, a linear pencil is given by a tuple
$(d;m;k;A_0,\ldots,A_d,A'_{1},\ldots,A'_{m})$ with $d,m,k\in\N$ and
$A_0,\ldots,A_d,A'_{1},\ldots,A'_{m}\in\Q^{k\times k}$ rational symmetric 
matrices such that the projected spectrahedron is given by 
$\pi(S) = \{ x\in\R^d\ |\ \exists y\in\R^m:\ A(x,y) \succeq 0\}$.


Containment questions for spectrahedra are connected to feasibility questions
of semidefinite programs in a natural way.
A \emph{Semidefinite Feasibility Problem} (SDFP) is defined as the following
decision problem (see, e.g.,~\cite{khachiyan1997,Ramana1997b}). 
\begin{align} 
\begin{split}\label{eq:sdfp}
  &\text{Given } d,k\in\N \text{ and rational symmetric } 
  k\times k\text{-matrices } A_0,A_1,\ldots,A_d ,\\
  &\text{decide whether there exists } x\in\R^d 
  \text{ such that } A(x) \succeq 0 .
\end{split}
\end{align}
Equivalently, one can ask whether the spectrahedron $S_A$ is nonempty.
Although checking positive semidefiniteness can be done in polynomial time by
computing a Cholesky factorization, the complexity classification of the
problem SDFP is one of the major open complexity questions related to
semidefinite programming (see~\cite{deklerk2002, Ramana1997b}).
Using semidefinite programming techniques, a SDFP can be solved efficiently in
practice.
In our model of computation, the binary Turing machine, SDFP is known to be
feasible in polynomial time if the number of variables $d$ or the matrix size
$k$ is fixed~\cite[Theorem 7]{khachiyan1997}. 

We first discuss the complexity classification concerning only projections of
polytopes.

\begin{thm}\label{thm:complexity_pih_h}
Deciding whether a projected $\HH$-polytope is contained in an $\HH$-polytope
can be done in polynomial time.
\end{thm}

\begin{proof}
Let $\pip = \{x\in\R^d\ |\ \exists y\in\R^m : a+Ax+A'y \geq 0\}$ be a
projected $\HH$-polytope and let $Q = \{x\in\R^d\ |\ b+Bx \geq 0\}$ be an
$\HH$-polytope. Embed $Q$ into $\R^{d+m}$ by 
$Q' = \{(x,y) \in\R^{d+m}\ |\ b+Bx+0y \geq 0\}$. Then the containment 
problem $\pip\subseteq Q$ is equivalent to the $\HH$-in-$\HH$ containment 
problem $P \subseteq Q'$. The statement then follows
from~\cite{gritzmann1994}.
\end{proof}

In the latter theorem, the statement does not differ from the non-projected
case. The next theorem shows a significant change in the complexity
classification when the outer set is a projected $\HH$-polytopes.

\begin{thm} \label{thm:complexity_h_pih}
Deciding whether an (projected) $\HH$-polytope is contained in a projected
$\HH$-polytope is co-NP-complete. 
\end{thm}

\begin{proof}
Consider a $\VV$-polytope. It has a representation as the projection of an
$\HH$-polytope polynomial in the input data. Thus the containment problem
$\HH$-in-$\pih$ is co-NP-hard since $\HH$-in-$\VV$ is co-NP-complete. It is
also in the class co-NP since given a certificate for '$\HH$ not in $\pih$',
i.e., a point $p$, one can test whether $p\in\HH$ and $p\not\in\pih$ by
evaluating the linear constraints of $\HH$ (all have to be satisfied) and by
solving a \emph{linear feasibility problem} which both is in P
by~\cite[Theorem 13.4]{Schrijver1986}.
Therefore $\HH$-in-$\pih$ is co-NP-complete. Obviously, the proof remains
valid when passing to $\pih$-in-$\pih$.
\end{proof}

In the remaining part, we study the complexity of containment problems
involving projections of spectrahedra.

As the complexity of SDFP is unknown, the subsequent statement on containment
of a spectrahedron in an $\HH$-polytope does not give a complete answer
concerning polynomial solvability of this containment question in the Turing
machine model. 

\begin{thm} \label{thm:complexity_s_h}
The problem of deciding whether the projection of a spectrahedron is contained
in an $\HH$-polytope can be formulated by the complement of semidefinite
feasibility problems (involving also strict inequalities), whose sizes are
polynomial in the description size of the input data.
\end{thm}

\begin{proof}
Consider a spectrahedron $S_A$ given by the linear matrix pencil $A(x,y)$ and 
the coordinate projection of $S_A$ onto the $x$-variables $\pisa$.
Given an $\HH$-polytope $P = \{x \in\R^d\ |\ b+Bx \ge 0\}$ with $b\in\Q^l$ and
$B\in\Q^{l\times d}$, construct for each $i\in [l]$ the SDFP
\[
  (b+Bx)_i < 0,\ A(x,y)\succeq 0
\]
involving a strict inequality.
Then $\pisa\not\subseteq P$ if one of the $l$ SDFPs is not solvable.
\end{proof}

While the $\pis$-in-$\HH$ containment problem is efficiently solvable in
practice, the situation changes if the outer set is given as the projection of
an $\HH$-polytope. 

\begin{thm}\label{thm:complexity_pis_pis}
\ 

\begin{enumerate}
  \item
  Deciding whether an (projected) $\HH$-polytope or a (projected)
  spectrahedron is contained in a (projected) spectrahedron is co-NP-hard. 
  \item
  Deciding whether a (projected) spectrahedron is contained in the projection
  of an $\HH$-polytope is co-NP-hard.
\end{enumerate}
\end{thm}

\begin{proof}
Since the problem $\HH$-in-$\sym$ is co-NP-hard (see~\cite[Proposition
4.1]{bental2002} and~\cite[Theorem 3.4]{ktt1}), deciding whether a projected
$\HH$-polytope or projected spectrahedron is contained in a (projected)
spectrahedron is co-NP-hard as well. This is part (1) of the theorem.

Parts (2) is a consequences of Theorem~\ref{thm:complexity_h_pih}.
\end{proof}

\subsection{Hol-Scherer's Positivstellensatz} \label{sec:hol}

Consider a symmetric matrix polynomial $G=G(x)\in\sym^k[x]$ in the variables
$x=(x_1,\ldots,x_d)$, i.e., a symmetric matrix whose entries lie in the
polynomial ring $\R[x]$. We say $G$ has degree $t$ if the maximum degree of
the entries is $t$, i.e., $ t = \max\{ \deg(G_{ij})\ |\ i,j\in [k] \}$.

For matrices $M = (M_{ij})_{i,j=1}^l\in\sym^{kl}$ and $N\in\sym^k$, define 
\begin{equation} \label{eq:scalar}
  \left\langle M,N \right\rangle_l 
  := \left( \left\langle M_{ij},N \right\rangle \right)_{i,j=1}^l 
  = \sum_{i,j=1}^l E_{ij} \cdot \left\langle M_{ij},N \right\rangle ,
\end{equation}
where $E_{ij}$ denotes the $l\times l$-matrix with one in the $(i,j)$th entry
and zero otherwise, and $\langle\cdot,\cdot\rangle$ is the Euclidean inner
product for matrices.
We refer to~\eqref{eq:scalar} as the \emph{$l$th scalar product}. It can be
seen as a generalization of the Gram matrix representation of a positive
semidefinite matrix. Indeed, for positive semidefinite matrices $M$ and $N$
the $l\times l$-matrix $\left\langle M,N \right\rangle_l$ is positive
semidefinite as well~\cite{Hol2004}.

For any positive integer $l$, define the quadratic module generated by $G(x)$
\begin{equation} \label{eq:qm-matrix2}
  \qm^l(G) = \left\{ S_0(x) + \left\langle S(x),G(x) \right\rangle_l\ |\
    S_0(x)\in\SOS^l[x],\ S(x)\in\SOS^{kl}[x] \right\} \subseteq \sym^l[x] ,
\end{equation}
where $\SOS^k[x]\subseteq\sym^k$ is the set of sum of squares 
$k\times k$-matrix polynomials.
A matrix polynomial $S=S(x)\in\sym^k[x]$ is called \emph{sum of squares}
(sos-matrix for short) if it has a decomposition $S = U(x)U(x)^T$ with
$U(x)\in\R^{k\times m}[x]$ for some positive integer $m$.
Equivalently, $S$ has the form $(I_k \otimes [x]_t)^T Z (I_k\otimes [x]_t)$,
where $[x]_t$ denotes the monomial basis in $x$ up to 
$t = \max\{\deg(S_{ij}(x))/2\ |\ i,j\in[k]\}$ and $Z$ is a positive
semidefinite matrix of appropriate size.
For $k=1$, $S$ is called a sos-polynomial.
Checking whether a matrix polynomial is a sos-matrix is an
SDFP~\eqref{eq:sdfp}. 
Obviously, every element in $\qm^l(G)$ is positive semidefinite on the
semialgebraic set $S_G := \{x\in\R^d\ |\ G(x)\succeq 0 \}$.
Hol and Scherer~\cite{Hol2006} showed that for matrix polynomials positive
definite on $S_G$ the converse is true under the Archimedeanness condition.

We state the desired Positivstellensatz of Hol and Scherer.
See~\cite{klep2010} for an alternative proof by Klep and Schweighofer using
the concept of pure states.

\begin{prop}[{\cite[Corollary 1]{Hol2006}}]
\label{prop:hol}
Let $l$ be a positive integer and let $S_G = \{x\in\R^d\ |\ G(x)\succeq 0 \}$
for a matrix polynomial $G\in\sym^k[x]$. If the quadratic module $\qm^l(G)$ is
Archimedean, then it contains every matrix polynomial $F\in\sym^l[x]$ positive
definite on $S_G$.
\end{prop}

By restricting to diagonal matrix polynomials $G$ and $l=1$, one gets the
Positivstellensatz of Putinar~\cite{putinar1993} as a corollary.

\begin{cor}\label{cor:putinar}
Let $G=\{g_1,\ldots,g_k\}\subseteq\R[x]$ and $S_G = \{x\in\R^d\ |\ g\ge 0\
\forall g\in G\}$.
If the quadratic module
\[
  M(G) = \left\{s_0(x) + \sum_{i=1}^k s_i(x) g_i(x)\ |\
s_0,s_1,\ldots,s_m\in\Sigma^1[x] \right\}
\]
is Archimedean, then it contains every polynomial $f\in\R[x]$ positive on
$S_G$.
\end{cor}

Interestingly, the usual quadratic module $\qm^1(G)$ is Archimedean if
and only if $\qm^l(G)$ is for any positive integer $l$.

\begin{prop}\label{prop:archimedean}
The following two statements are equivalent.
\begin{enumerate}
  \item
  For some positive integer $l$, the quadratic module $\qm^l(G)$ is 
Archimedean.
  \item
  For all positive integers $l$, the quadratic module $\qm^l(G)$ is 
Archimedean.
\end{enumerate}
Furthermore, assume $G$ is a linear pencil.
Then $\qm^l(G)$ for any positive integer $l$ is Archimedean if and only if the
spectrahedron $S_G$ is bounded.
\end{prop}

The equivalence of the first two statements was proved by Helton, Klep, and
McCullough for monic linear matrix pencils in the language of their
matricial relaxation; see~\cite[Lemma 6.9]{Helton2010}. We recapitulate the
proof and extend it to quadratic modules generated by arbitrary matrix
polynomials.

\begin{proof}
The implication $(2)\Longrightarrow(1)$ is obvious.
To show the reverse implication, note first that $\qm^l(G)$ is Archimedean if
and only if $(N-x^T x)I_l \in\qm^l(G)$ for some positive integer $N$.
Let $m\in\N$ be arbitrary but fixed.
We have to show that $(N-x^T x)I_m \in\qm^m(G)$.
Denote by $E_{11}$ the $m\times m$-matrix with one in the entry $(1,1)$ and
zero elsewhere and let $Q$ be the $l\times m$-matrix with one in the entry
$(1,1)$ and zero elsewhere.
Clearly, $E_{11} = Q^T Q$. 
Let $ (N-x^T x)I_l = S_0 + \left\langle S,G \right\rangle_l$ with
$S=(S_{ij})_{i,j=1}^l$ be the desired sos-representation.
Setting $\tilde{S}_0 := Q^T S_0 Q = (S_0)_{11} E_{11} \in\SOS^m[x]$ and 
$\tilde{S} = E_{11} \otimes S_{11} \in\SOS^m[x]$, we get
\[
  (N-x^T x)E_{11} = Q^T (N-x^T x)I_l Q 
  = Q^T(S_0 + \left\langle S,G \right\rangle_l)Q
  = \tilde{S}_0 + E_{11} \left\langle S_{11},G \right\rangle
  = \tilde{S}_0 + \langle \tilde{S},G \rangle_m .
\]
Applying the same to $E_{ii}$ for $i\in [m]$ and using additivity of the
quadratic module $\qm^m(G)$ yields $(N-x^T x)I_m \in\qm^m(G)$.

The last statement follows from~\cite[Corollary 4.4.2]{klep2013} (see
also~\cite{Klep2011}) together with the shown equivalence.
\end{proof}

\section{A Bilinear Formulation of the
\texorpdfstring{$\pih$}{piH}-in-\texorpdfstring{$\pih$}{piH} Containment
Problem} \label{sec:pihpih}

For $a\in\R^k,\ A\in\R^{k\times d},\ A'\in\R^{k\times m}$ and 
$b\in\R^l,\ B\in\R^{l\times d},\ B'\in\R^{l\times n}$ let
\begin{align}
\begin{split} \label{eq:projpoly}
  \pipa &= \left\{x\in\R^d\ |\ \exists y\in\R^m :\ a+Ax+A'y\geq 0\right\} \\
  \text{and}\quad 
  \pipb &= \left\{x\in\R^d\ |\ \exists y'\in\R^n:\ b+Bx+B'y'\geq 0 \right\} \\
\end{split}
\end{align}
be projections of the $\HH$-polyhedra $P_A$ and $P_B$, respectively. 
Note that both $\pipa$ and $\pipb$ are $\HH$-polyhedra themselves (and thus
closed sets).
A quantifier-free $\HH$-description however can be exponential in the input
size $(d,m,k)$ respectively $(d,n,l)$; cf. Section~\ref{sec:prelim}.

Our starting point is the formulation of the containment problem as a bilinear
feasibility problem.
Interestingly, the projection variables $y'$ of the outer polyhedron do not
appear in the feasibility system (or the optimization version below) only the
corresponding coefficients $B'$.

\begin{thm} \label{thm:pihpih}
Let $\pipa$ and $\pipb$ be as defined in~\eqref{eq:projpoly} and $\pipa$ be 
nonempty.
\begin{enumerate}
  \item
  $\pipa$ is contained in $\pipb$ if and only if
\[
  z^T (b+Bx) \geq 0\ \text{ on }\ 
  \pipa\times\left(\kernel (B'^T)\cap\R^l_+\right) .
\]
  \item
  Let $\kernel (B'^T)\cap\R^l_+ = \linspan (B')^{\bot}\cap\R^l_+ \neq \{0\}$.
  Then $\pipa\subseteq\pipb$ if and only if
\[
  z^T (b+Bx) \geq 0\ \text{ on }\ 
  \pipa\times \left(\kernel(B'^T)\cap\Delta^l\right) ,
\]
  where $\Delta^l = \{z\in\R^l\ |\ \mathds{1}_l^T z = 1,\ z\geq 0\}$ is the
$l$-simplex.
\end{enumerate}
\end{thm}

The additional assumption on the kernel of $B'^T$ seems to be somewhat
artificial, however, if the projection of $P_B$ to the $x$-coordinates is
bounded, then the condition holds.
The two main advantages of part (2) in Theorem~\ref{thm:pihpih} are the
boundedness of the $z$ variables and that the condition $z^T (b+Bx) \geq 0$ is
indeed an inequality. 
(Note that in part (1), containment is equivalent to $z^T (b+Bx)\equiv 0$ on
$\pipa\times(\kernel(B'^T)\cap\R^l_+)$ as $(x,z)=(x,0)$ is a feasible
solution for all $x\in\pipa$.)
The next lemma serves as a first step in a geometric interpretation of this
precondition.

\begin{lemma}\label{lem:kernel}
Let $\pipb$ be as in~\eqref{eq:projpoly}.
Then $\kernel(B'^T)\cap\R^l_+ = \linspan(B')^{\bot}\cap\R^l_+ = \{0\}$ 
if and only if $\linspan(B')\cap\R^l_{++}\neq\emptyset$. 
In this case, $\pipb=\R^d$.

In particular, if $\pipb$ is bounded, then 
$\kernel(B'^T)\cap\R^l_+ = \linspan (B')^{\bot}\cap\R^l_+ \neq \{0\}$.
\end{lemma}

\begin{proof}
The equivalence 
$\kernel(B'^T)\cap\R^l_+ = \linspan (B')^{\bot}\cap\R^l_+ = \{0\} \iff 
\linspan (B')\cap\R^l_{++}\neq\emptyset$ is easy to see. 
If so, then there exists $y'\in\R^n$ such that $B' y' > 0$. Thus, for every 
$x\in\R^d$, there exists $t>0$ sufficiently large such that 
$b+Bx+B'(ty')\geq 0$. This implies $\pipb=\R^d$.
Thus for bounded $\pipb$ we have $\kernel(B'^T)\cap\R^l_+ \neq\{0\}$.
\end{proof}

Before proving Theorem~\ref{thm:pihpih}, we observe that neither the
implication
``$\linspan (B')\cap\R^l_{++}\neq\emptyset \Longrightarrow \pipb=\R^d$''
nor the implication 
``$\pipb$ is bounded $\Longrightarrow \kernel(B'^T)\cap\R^l_+ \neq\{0\}$'' 
in Lemma~\ref{lem:kernel} is an equivalence. 
Example~\ref{ex:proj1} also shows that the precondition in part (2) of 
Theorem~\ref{thm:pihpih} cannot be dropped.

\begin{ex}\label{ex:proj1}
\ 

  (1)
Consider the polyhedron 
\[
  P_1 = \left\{\begin{pmatrix} x \\ y \end{pmatrix} \in\R^2\ |\
  \begin{pmatrix} 1 \\ 1 \end{pmatrix} + \begin{pmatrix} -1 \\ 1 \end{pmatrix} 
  x + \begin{pmatrix} 1 \\ 1 \end{pmatrix} y \geq 0 \right\} .
\]
$P_1$ is a pointed polyhedral cone containing the origin in its interior; see
Figure~\ref{fig:proj-ex1} (A).
We have $\linspan (B')\cap\R^2_{++}\neq\emptyset$ and thus the intersection 
of $\kernel(B'^T) = \kernel(1,1)$ and the nonnegative real numbers is 
zero-dimensional, i.e., $\kernel(B'^T)\cap\R^2_+=\{0\}$.
Moreover, in this case, the restriction to the $1$-simplex as in part (2) of
Theorem~\ref{thm:pihpih} is not possible.

\begin{figure}
\centering
\begin{subfigure}{.35\textwidth}
  \centering
  \includegraphics[width=.75\linewidth]{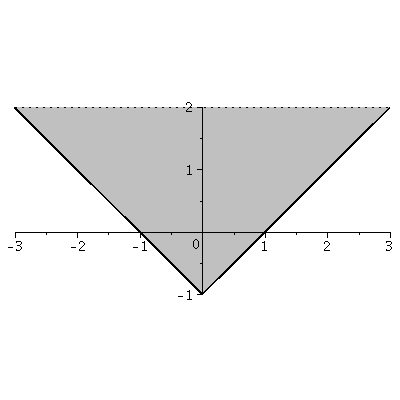}
  \caption{$P_1$ as defined in Example~\ref{ex:proj1}.}
\end{subfigure} \hspace*{.1\textwidth}
\begin{subfigure}{0.35\textwidth}
  \centering
  \includegraphics[width=.75\linewidth]{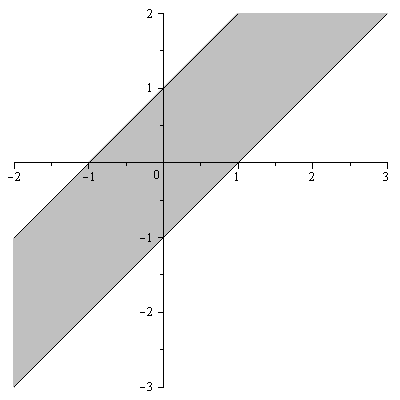}
  \caption{$P_2$ as defined in Example~\ref{ex:proj1}.}
\end{subfigure}
\caption{}  
\label{fig:proj-ex1}
\end{figure}

  (2)
Consider the polyhedron
\[
  P_2 = \left\{\begin{pmatrix} x \\ y \end{pmatrix} \in\R^2\ |\
  \begin{pmatrix} 1 \\ 1 \end{pmatrix} 
  + \begin{pmatrix} 1 \\ -1 \end{pmatrix} x 
  + \begin{pmatrix} -1 \\ 1 \end{pmatrix} y \geq 0 \right\} ,
\]
which is unbounded and contains the origin in its interior; see
Figure~\ref{fig:proj-ex1} (B).
We have $\linspan(B')\cap\R^2_{++}=\emptyset$ and thus 
$\kernel(B'^T)\cap\R^2_+ \neq\{0\}$. Indeed, for every $t\geq 0$, we have 
$(0,t)\in\kernel(B'^T)\cap\R^2_+$.
On the other hand, $\pipb=\R$ shows that the reverse of the other (and above
mentioned) implications in Lemma~\ref{lem:kernel} are not equivalences.
\end{ex}

\begin{proof}[Proof of Theorem~\ref{thm:pihpih}]
\ 

(1):
$\pipa\not\subseteq\pipb$ if and only if there exists a point 
$x\in\pipa\bs\pipb$, i.e., for $x\in\pipa$ there exists no $y'\in\R^n$
with $b+Bx+B'y'\geq 0$.
By Farkas' Lemma~\cite[Proposition 1.7]{Ziegler1995} this is equivalent to the
existence of a point $z\in\R^l_+$ with $z^T B' = 0$ such that $z^T (b+Bx)<0$
holds.
Equivalently, there exists $(x,z)\in \pipa\times(\kernel(B'^T)\cap\R^l_+)$
such that $ z^T (b+Bx) < 0 $. 

(2):
If there exists
$(x,z)\in\pipa\times\left(\kernel(B'^T)\cap\R^l_+\right)$ such that 
$z^T (b+Bx)<0$, then $z\neq 0$ and thus $z'^T (b+Bx) < 0$ for
$z'=\frac{z}{|z|}\geq 0$ with 
$|z'|=\sum_{i=1}^l z'_i = \frac{1}{|z|} \sum_{i=1}^l z'_i = 1$.

Assume $z^T (b+Bx) \geq 0$ holds for all
$(x,z)\in\pipa\times\left(\kernel(B'^T)\cap\R^l_+\right)$. 
By assumption, there exists $0\neq z\in\kernel(B'^T)\cap\R^l_+$. 
Applying the same scaling as above yields $z^T (b+Bx) \geq 0$ for every
$z\in\Delta^l$, implying the claim.
\end{proof}

Note that $\kernel(B'^T)\cap\Delta^l_+$ is a polytope and is intrinsically
linked to the polar of $\pipb$. Namely, it is the set of convex combinations
of the columns in $B'^T$ that are equal to the origin, i.e., $0=B'^T z$ with
$1=\mathds{1}_l^T z$ and $z\geq 0$.

Consider the optimization version of Theorem~\ref{thm:pihpih}
\begin{align}
\begin{split} \label{eq:pihpih}
  \inf\ &\ z^T(b+Bx) \\
   \st\ &\ (x,y,z)\in P_A\times \left( \kernel (B'^T)\cap\Delta^l_+ \right) .
\end{split}
\end{align}
Assuming nonemptyness of $\kernel (B'^T)\cap\Delta^l_+$,
Theorem~\ref{thm:pihpih} implies that $\pipa\subseteq\pipb$ if and only if the
infimum is nonnegative.

Replacing the nonnegativity constraints in~\eqref{eq:pihpih} by sos
constraints results in a hierarchy of SDFPs to decide the $\pih$-in-$\pih$
containment problem,
\begin{align}
\begin{split} \label{eq:pihpihopt}
  \mu(t) := \sup\ &\ \mu \\
   \st\ &\ z^T(b+Bx) -\mu \in M^1 + I ,\\
\end{split}
\end{align}
where $M^1$ and $I$ denote the quadratic module generated by the inequality
constraints and the ideal generated by the equality constraints, respectively.

Under the assumptions in part (2) of Theorem~\ref{thm:pihpih}, applying
Putinar's Positivstellensatz, Corollary~\ref{cor:putinar}, to
problem~\eqref{eq:pihpihopt} (also allowing equality constraints), the
sequence $\mu(t)$ convergences asymptotically to the optimal value
of~\eqref{eq:pihpih} for $t\to\infty$.

\section{A Bilinear Formulation of the
\texorpdfstring{$\pis$}{piS}-in-\texorpdfstring{$\pis$}{piS} Containment
Problem} \label{sec:pispis}

Throughout the section, let 
\begin{align*}
  A(x,y) &= A_0 + \sum_{i=1}^d A_i x_i + \sum_{j=1}^m A'_j y_j \in\sym^k[x,y]
  \\ \text{ and }\quad 
  B(x,y') &= B_0 +\sum_{i=1}^d B_i x_i + \sum_{j=1}^n B'_j y'_j \in\sym^l[x,y']
\end{align*}
be linear pencils with $y=(y_1,\ldots,y_m)$ and $y'=(y'_1,\ldots,y'_n)$ for 
$n\geq 1$. Denote the projection of the corresponding spectrahedra onto the 
$x$-variables by 
\begin{align*}
  \pisa &= \left\{x\in\R^d\ |\ \exists y\in\R^m:\ A(x,y) \succeq 0 \right\}\\ 
  \text{and}\quad
  \pisb &= \left\{x\in\R^d\ |\ \exists y'\in\R^n:\ B(x,y')\succeq 0 \right\} .
\end{align*}
Recall from Section~\ref{sec:prelim} that the projection of a spectrahedron is 
not necessarily closed and thus, in general, not a spectrahedron itself.

Define $\bar{\Bcal} = \linspan\{B'_{1},\ldots,B'_{n}\}$ and recall the
equivalence
\[
  \langle B'_{i},Z \rangle = 0\ \forall i\in [n] \iff Z\in\bar{\Bcal}^{\bot} .
\]

The $\pis$-in-$\pis$ containment problem is slightly more involved than the
$\pih$-in-$\pih$ problem as the projection of a spectrahedron fails to be
closed in general. 
We state an extension of Theorem~\ref{thm:pihpih} to the $\pis$-in-$\pis$
containment problem.

\begin{thm}\label{thm:pispis}
Let $A(x,y)\in\sym^k[x,y]$ and $B(x,y')\in\sym^l[x,y']$ be linear pencils
such that $\pisa\neq\emptyset$. 
\begin{enumerate}
  \item
  $\pisa\subseteq \cl\pisb$ if and only if
  $\langle B(x,0), Z \rangle\geq 0$ on 
  $\pisa\times ( \bar{\Bcal}^{\bot}\cap\psd^l )$.
  \item
  Assume that the condition 
  $\sum_{i=1}^{n} B'_i y'_i \succeq 0 \Longrightarrow 
  \sum_{i=1}^{n} B'_i y'_i = 0$ holds for all $y'$.
  Then the closure in part (1) can be dropped.
\end{enumerate}
\end{thm}

As in the $\pih$-in-$\pih$ problem, the projection variables $y'$ of the outer
spectrahedron do not appear in the polynomial formulation, only the
corresponding coefficient matrices.

We use the following Farkas type lemmas to prove Theorem~\ref{thm:pispis}.

\begin{lemma}[{\cite[Theorem 2.22]{tuncel2010}}] \label{lem:farkas_cone_2}
Let $A(x)\in\sym^k[x] $ be a linear pencil and denote by 
$\tilde{A}(x)=\sum_{i=1}^d x_i A_i $ the pure-linear part. 
Then exactly one of the following two systems has a solution.
\begin{align}
  \forall \varepsilon > 0\ \exists A'_0 \in\sym^k,\ \exists x \in\R^d &:\ 
  \| A_0 - A'_0 \| < \varepsilon,\ A'_0 + \tilde{A}(x) \in\psd^k \\
  \exists Z \in\sym^k &:\ Z\succeq 0,\ \left\langle A_i,Z\right\rangle = 0\  
  \forall i\in [d] ,\ \left\langle A_0 , Z \right\rangle < 0
\label{eq:farkas2-2}
\end{align}
\end{lemma}

The Farkas type lemmas for cones (and thus the theory of semidefinite
programming) lack in the fact that the linear image of the cone of positive
semidefinite matrices is not closed in general. 
Additional conditions which lead to more clean formulations are called
\emph{constraint qualification}.

\begin{lemma}[{\cite[Example 5.14]{Boyd2004}}] \label{lem:boyd}
Let $A(x)\in\sym^k[x]$ be a linear pencil. Assume
\[
  \sum_{i=1}^d A_i x_i \succeq 0 \Longrightarrow \sum_{i=1}^d A_i x_i = 0
\]
holds for any $x$.
Then either \eqref{eq:farkas2-2} has a solution or $S_A$ is nonempty.

If $A_1,\ldots,A_d$ are linearly independent, then the above condition can be
replaced by $\sum_{i=1}^d A_i x_i \succeq 0 \Longrightarrow x=0 $.
\end{lemma}

\begin{proof}[Proof of Theorem~\ref{thm:pispis}]
\ 

(1):
Assume $\pisa\subseteq\cl\pisb$. Let $x\in\pisa$.
Then there exists a sequence $(x_i,y'_i)_i \subseteq S_B$ such that
$\lim_{i\to\infty} x_i = x$.
For all $Z\in\bar{\Bcal}^{\bot}\cap\psd^l$ it holds that
\[
  \left\langle B(x,0),Z \right\rangle 
  = \lim_{i\to\infty} \left\langle B(x_i,y'_i),Z \right\rangle \geq 0 .
\]
Since $x\in\pisa$ is arbitrary, $\langle B(x,0),Z \rangle$ is nonnegative
on $\pisa\times(\bar{\Bcal}^{\bot}\cap\psd^l)$.

Assume $\left\langle B(x,0),Z \right\rangle\geq 0$ on
$\pisa\times(\bar{\Bcal}^{\bot}\cap\psd^l)$. 
Let $x\in\pisa$ be fixed but arbitrary and set $B'_0=B(x,0)$.
By Lemma~\ref{lem:farkas_cone_2}, there exist $B''_0 \in\sym^l$ and
$y'\in\R^n$ such that $B''_0 + \sum_{i=1}^n B'_i y'_i\in\psd^l$ and 
$\| B'_0 -B''_0 \| < \varepsilon$ for all $\varepsilon > 0$. By letting
$\varepsilon$ tend to zero, there exists a sequence
$(y'_{\varepsilon})_{\varepsilon}\subseteq\R^n$ such that
$\lim_{\varepsilon\to 0}B(x,y'_{\varepsilon})\succeq 0$.
As $x\in\pisa$ is arbitrary, the claim follows.

(2):
Assume $\left\langle B(x,0),Z \right\rangle\geq 0$ on
$\pisa\times(\bar{\Bcal}^{\bot}\cap\psd^l)$. 
Let $x\in\pisa$ be fixed but arbitrary. 
By Lemma~\ref{lem:boyd}, the spectrahedron 
$\{y'\in\R^n\ |\ B'_0 + \sum_{i=1}^n B'_{i} y'_i \succeq 0\}$ is nonempty. 
Thus there exists $y'\in\R^n$ such that $B(x,y')\succeq 0$. 
\end{proof}

Unfortunately, the if-part in Theorem~\ref{thm:pispis} (1) without taking the
closure is generally not true as the next example shows.

\begin{ex} \label{ex:pispispos}
Consider the linear pencil 
\[
  B(x,y') = \begin{bmatrix} -y'_1 & x & 0 \\ x & 1-y'_2 & 0 \\ 0 & 0 & -x+y'_2
\end{bmatrix}
  = \begin{bmatrix} 0 & x & 0 \\ x & 1 & 0 \\ 0 & 0 & -x \end{bmatrix}
  + y'_1 \begin{bmatrix} -1 & 0 & 0 \\ 0 & 0 & 0 \\ 0 & 0 & 0 \end{bmatrix}
  + y'_2 \begin{bmatrix} 0 & 0 & 0 \\ 0 & -1 & 0 \\ 0 & 0 & 1 \end{bmatrix} 
\]
and let $A(x)$ be the univariate linear pencil
\[
  A(x) = \begin{bmatrix} 1-x & 0 \\ 0 & 1+x \end{bmatrix}
  = \begin{bmatrix} 1 & 0 \\ 0 & 1 \end{bmatrix}
  + x \begin{bmatrix} -1 & 0 \\ 0 & 1 \end{bmatrix}
\]
describing the interval $S_A = [-1,1]$. By inspecting the principal minors of 
$B$, the spectrahedron $S_B$ has the form
$ \{(x,y')\in\R^3\ |\ y'_1\leq 0,\ x\leq y'_2\leq 1,\ 
  y'_1(1-y'_2)+x^2 \leq 0\} $.
For $x= 1$, the second condition implies $y'_2=1$ and thus the third condition 
reads as $x^2\leq 0$, a contradiction.
Thus $S_A \not\subseteq \pisb = (-\infty,1)$.

For every $Z\in\bar{\Bcal}^{\bot}\cap\psd^3$ it holds that
\begin{align*}
  0 = \left\langle Z,B'_1 \right\rangle = -Z_{11} \Longrightarrow Z_{12}=0,
  \quad 0 = \left\langle Z,B'_2 \right\rangle = Z_{33} - Z_{22} 
\end{align*}
implying 
$ \left\langle B(x,0),Z\right\rangle=Z_{22}+x(-Z_{33}+2Z_{12}) 
  =Z_{22}(1-x)\geq 0 $ for all $x\in S_A$. 

It should not be surprising that the constraint qualification on the pencil 
$B(x,y')$ is not satisfied. Indeed, for $(y'_1,y'_2)=(y'_1,0)$ with $y'_1<0$, 
\[
  B'_1 y'_1 + B'_2 y'_2 
  = \begin{bmatrix} -y'_1 & x & 0 \\ x & -y'_2 & 0 \\ 0 & 0 & -x+y'_2
\end{bmatrix}
  = \begin{bmatrix} -y'_1 & 0 & 0 \\ 0 & 0 & 0 \\ 0 & 0 & 0 \end{bmatrix}
\]
is positive semidefinite but not identically zero.
\end{ex}
%

An issue when considering the practical utility of Theorem~\ref{thm:pispis} is
the unboundedness of the set $\bar{\Bcal}^{\bot}\cap\psd^l$. 
Under an analog condition as in Theorem~\ref{thm:pihpih}, $\psd^l$ can be
replaced by the spectrahedral analog of the simplex. 

\begin{cor} \label{cor:pispis}
Let $A(x,y)\in\sym^k[x,y]$ and $B(x,y')\in\sym^l[x,y']$ be linear pencils
such that $\pisa\neq\emptyset$.
Assume $\bar{\Bcal}^{\bot}\cap\psd^l\neq\{0\}$.
Then $\pisa\subseteq\cl\pisb$ if and only if 
\[
  \langle B(x,0),Z \rangle\geq 0 \text{ on } 
  \pisa\times\left(\bar{\Bcal}^{\bot}\cap\T^l\right) ,
\]
where $\T^l = \{ Z\in\psd^l\ |\ \langle I_l,Z\rangle=1 \}$ is the
$l$-spectraplex.
\end{cor}

\begin{proof}
Since $\bar{\Bcal}^{\bot}\cap\T^l\subseteq\bar{\Bcal}^{\bot}\cap\psd^l$, the
``only if''-part follows from Theorem~\ref{thm:pispis}.

For the converse, first suppose there exists 
$(x,Z)\in\pisa\times(\bar{\Bcal}^{\bot}\cap\psd^l)$ such that 
$\langle B(x,0),Z\rangle < 0$. 
Then $0\neq Z\in\psd^l$ and thus $\tr(Z) = \langle I_l,Z\rangle > 0$. 
This implies $\langle B(x,0),Z'\rangle < 0$ for $Z' = \frac{Z}{\tr(Z)}$ with 
$\tr(Z')=\langle I_l,Z'\rangle = \frac{1}{\tr(Z)} \langle I_l,Z\rangle =1$.

Assume $\langle B(x,0),Z\rangle\geq 0$ on
$\pisa\times(\bar{\Bcal}^{\bot}\cap\psd^l)$.
By assumption, there exists 
$0\neq Z\in\bar{\Bcal}^{\bot}\cap\psd^l=\bar{\Bcal}^{\bot}\cap\psd^l$. 
Applying the above scaling, the claim follows.
\end{proof}

We state an analogue to Lemma~\ref{lem:kernel}. As the proof is very similar,
we skip it here. 

\begin{lemma}
We have $\bar{\Bcal}^{\bot}\cap\psd^l = \{0\}$ if and only if
$\mathcal{B}\cap\pd^l\neq\emptyset$. In this case, $\pisb=\R^d$.

In particular, if $\pisb$ is bounded, then
$\bar{\Bcal}^{\bot}\cap\psd^l\neq\{0\}$.
\end{lemma}

Restricting the $\pis$-in-$\pis$ containment problem to the special case 
$\pis$-in-$\pih$ allows to state improved versions of Theorem~\ref{thm:pispis}
and Corollary~\ref{cor:pispis}.

\begin{prop}
Let $\pipb$ be as in~\eqref{eq:projpoly} and let $A(x,y)\in\sym^l[x,y]$ be a
linear pencil. 
\begin{enumerate}
  \item
  $\pisa\subseteq\pipb$ if and only if $z^T (b+Bx)\geq 0$ on 
  $\pisa\times(\kernel(B'^T)\cap\R^l_+)$.
  \item
  Assume $\kernel(B'^T)\cap\R^l_+\neq\{0\}$.
  Then $\pisa\subseteq\pipb$ if and only if $z^T (b+Bx)\geq 0$ on 
  $\pisa\times(\kernel(B'^T)\cap\Delta^l)$.
\end{enumerate}
\end{prop}

\begin{proof}
$\pisa\not\subseteq\pipb$ if and only if there exists $x\in\pisa$ such that 
$\nexists y'\in\R^n :\ b+Bp+B'y'\geq 0$. By Farkas' Lemma~\cite[Proposition
1.7]{Ziegler1995} this is equivalent to the existence of a $z\in\R^l_+$
with $z^T B'=0$ and $z^T (b+Bp)<0$. The claims follow as in the proofs of
Theorem~\ref{thm:pispis} and Corollary~\ref{cor:pispis}.
\end{proof}

We close with an example.

\begin{ex}\label{ex:pispis}
Let $M$ be the convex hull of the shifted unit disks defined by the identities
$1-(x_1+1)^2-x_2^2 =0$ and $1-(x_1-1)^2-x_2^2=0$, respectively.
$M$ is the projection of a spectrahedron.
Indeed, considering only the first disk and shifting it along the segment
$[-1,1]\times\{0\}$ yields 
$ M = \left\{ x\in\R^2\ |\ \exists y\in\R:\ 1-(x-y)^2-x_2^2 \geq 0,\ 
  -1\leq y\leq 1 \right\} $.
It is the projection of the 3-dimensional cylinder, see
Figure~\ref{fig:cylinder}, defined by the linear pencil
\begin{align*}
  A(x,y) = \begin{bmatrix} 1-x_2 & x_1-y \\ x_1-y & 1+x_2 \end{bmatrix} \oplus
  \begin{bmatrix} 1-y&0 \\ 0& 1+y \end{bmatrix} .
\end{align*}

\begin{figure}[ht]
\centering
\includegraphics[width=.25\textwidth]{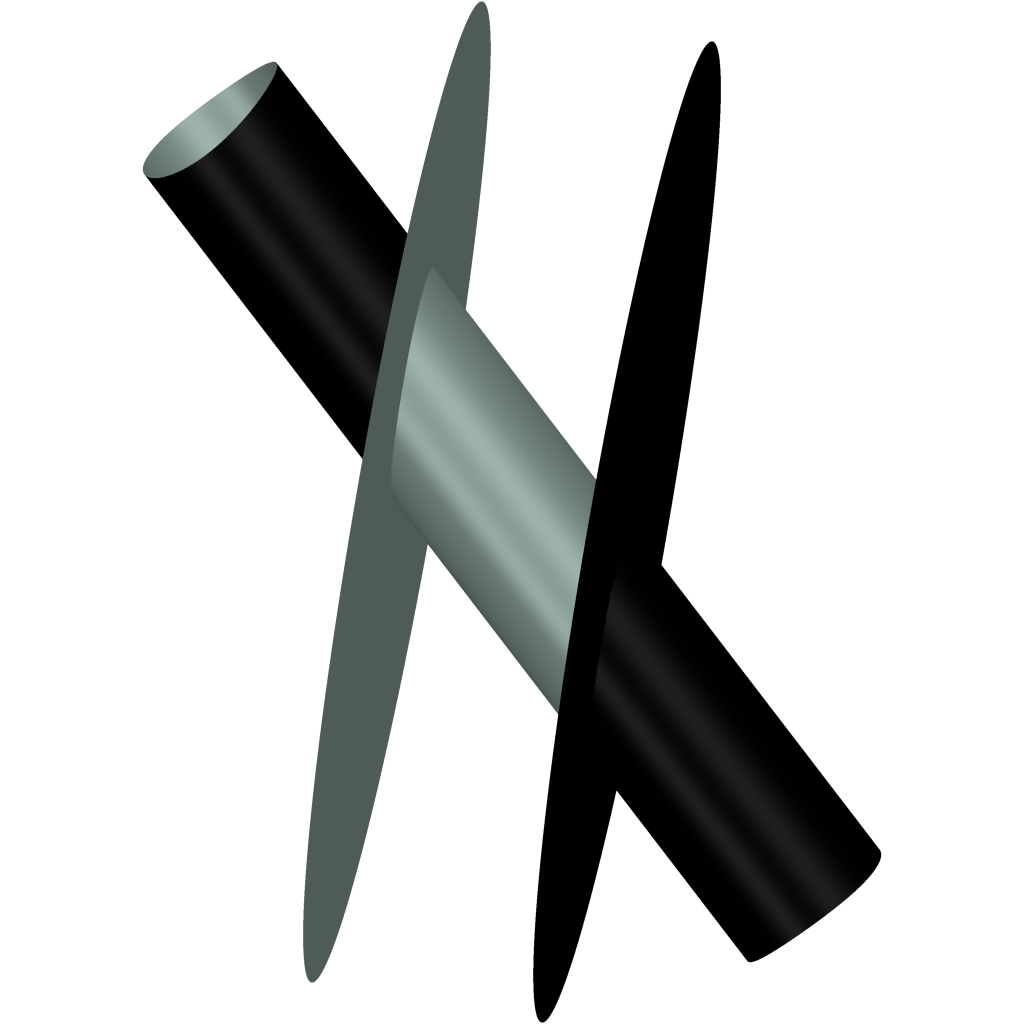}
\caption{The determinantal variety of $A(x,y)$ with $S_A$ being the grey
cylinder in the middle of the picture; see Example~\ref{ex:pispis}.}
\label{fig:cylinder}
\end{figure}

The so-called TV screen (see, e.g., \cite[Section 6.3.1]{Blekherman2013}) is
the projection of the spectrahedron
\[
  S_B = \left\{ (x,y)\in\R^{2+2}\ |\ A(x,y) = 
  \begin{bmatrix} 1 + y_1 & y_2 \\ y_2 & 1 - y_1 \end{bmatrix} \oplus
  \begin{bmatrix} 1 & x_1 \\ x_1 & y_1 \end{bmatrix} \oplus 
  \begin{bmatrix} 1 & x_2 \\ x_2 & y_2 \end{bmatrix} \succeq 0 \right\} ,
\]
onto the $x$ variables; see Figure~\ref{fig:tvscreen}.

\begin{figure}[ht]
  \centering
  \includegraphics[scale=.25]{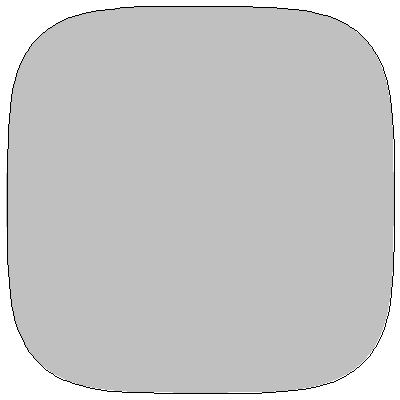}
  \caption{The TV screen as stated in Example~\ref{ex:pispis}.}
\label{fig:tvscreen}
\end{figure}

Both $\pisa$ and $\pisb$ are closed but not spectrahedra. We have
\begin{align*}
  \bar\B^{\bot}\cap\T^6
  &= \left\{ Z\in\psd^6\ |\ \sum_{i=1}^6 Z_{ii} = 1,\ Z_{22} = Z_{11} +
Z_{44},\ Z_{66}=-2Z_{12}\right\} .
\end{align*}
For all $Z\in \bar\B^{\bot}\cap\T^6$ the objective $\langle B(x,0),Z\rangle$
in Corollary~\ref{cor:pispis} can then be written as
\begin{align*}
  \left\langle B(x,0),Z\right\rangle
  &= 1-Z_{44}-Z_{66}+2x_1 Z_{34} + 2x_2 Z_{56} .
\end{align*}
We want to find a pair $(x^*,Z^*)\in\pisa\times (\bar\B^{\bot}\cap\T^6)$ such
that $\langle B(x,0),Z\rangle <0$.
To this end, consider $ x^* = (1+\varepsilon,0)\in\pisa$ for all 
$\varepsilon\in [0,1]$ and
\[
  Z^* = \begin{bmatrix} 0 & 0 \\ 0 & \frac{1}{3} \end{bmatrix} \oplus
  \begin{bmatrix} \frac{1}{3} & -\frac{1}{3} \\ -\frac{1}{3} & \frac{1}{3}
\end{bmatrix} \oplus 0_{3\times 3} \in\psd^6 .
\]
Then $\langle B(x^*,0),Z^*\rangle = -\frac{2}{3}\varepsilon <0$ for all
$\varepsilon\in (0,1]$. Thus $\pisa\not\subseteq\pisb$.

Now interchange the roles of $S_A$ and $S_B$, i.e., $\pisa$ is the TV screen
and $\pisb$ is the convex hull of two disks.
Then $\bar\B^{\bot}\cap\T^4$ is the set
\begin{align*}
  \bar\B^{\bot}\cap\T^4 
  &= \left\{ Z\in\psd^4\ |\ \sum_{i=1}^4 Z_{ii} = 1,\ 2 Z_{12} = Z_{44} -
Z_{33} \right\} .
\end{align*}
For all $Z\in \bar\B^{\bot}\cap\T^4$, $\langle B(x,0),Z\rangle$ has the form
\begin{align*}
  \left\langle B(x,0),Z\right\rangle
  &= (1-x_2)Z_{11} + (1+x_2)Z_{22} + (1-x_1)Z_{33} +(1+x_1)Z_{44} .
\end{align*}
As $1\pm x_i \ge 0$ for all $x\in\pisa$, we have $\langle B(x,0),Z\rangle\ge 0$
for all $(x,Z)\in\pisa\times (\bar\B^{\bot}\cap\T^4)$. Thus
$\pisa\subseteq\pisb$.
\end{ex}

\section{Sum of Squares Certificates for the
\texorpdfstring{$\pis$}{piS}-in-\texorpdfstring{$\sym$}{S} Containment Problem}
\label{sec:pis-s}

Retreating to the cases $\pih$-in-$\HH$ and $\pis$-in-$\sym$ allows to bring
forward several results from the non-projected case. 
We start with the polyhedral situation in Theorem~\ref{thm:pih-h}. 
It also serves as an algorithmic proof of Theorem~\ref{thm:complexity_pih_h}.
Afterwards, we state and prove a sophisticated Positivstellensatz for the
second problem.

\subsection{From the \texorpdfstring{$\pih$}{piH}-in-\texorpdfstring{$\HH$}{H}
to the \texorpdfstring{$\pis$}{piS}-in-\texorpdfstring{$\sym$}{S} Containment
Problem}
\label{sec:pih-to-pis}

As the proofs of the statements in this section are similar to the ones 
given in~\cite{ktt1}, we only stress the emerging differences in the proofs.

Even in the non-projected case, i.e., $m=0$, Theorem~\ref{thm:pih-h} below is
a slight extension of a statement in~\cite{ktt1}.
Namely, here we drop the conditions $a=\mathds{1}_k,$ and $b=\mathds{1}_l$ as
well as the boundedness condition.

\begin{thm} \label{thm:pih-h}
Consider the polyhedra $P_A = \{(x,y)\in\R^{d+m}\ |\ a+Ax+A'y \ge 0
\}\neq\emptyset$ and $P_B = \{x\in\R^d\ |\ b+Bx \ge 0 \}$.
\begin{enumerate}
  \item
  $\pipa\subseteq P_B$ if and only if there exists a nonnegative
  matrix $C\in\R^{l\times k}_+$ and a nonnegative vector $c_0\in\R^l_+$ with
  $b = c_0 + Ca,\ B = CA$, and $0=CA'$.
  \item
  Let $P_A$ be a polytope that is not a singleton. Then $\pipa\subseteq P_B$
  if and only if there exists a nonnegative matrix 
  $C\in\R^{l\times k}_+$ with $b = Ca,\ B=CA$, and $0=CA'$.
\end{enumerate}
\end{thm}

Testing whether $P_A$ is a singleton is easy as one has to check that the
system of equalities $a+Ax=0$ has a single solution. Certainly, in this
situation, checking containment is trivial as $\pipa\subseteq P_B$ is
equivalent to test whether a single point has nonnegative entries. 
The precondition in part (2) of Theorem~\ref{thm:pih-h}, however, cannot be
removed in general; see part (1) of Example~\ref{ex:h-h}.

For unbounded polyhedra the additional term $c_0$ is required in order for the
criterion to be exact. Without it, already in the simple case of two half
spaces defined by two parallel hyperplanes, the restriction of the condition
in part (1) of Theorem~\ref{thm:pih-h} to part (2) can fail to be feasible;
see part (2) of Example~\ref{ex:h-h}.

\begin{ex}\label{ex:h-h}
\ 

  (1)
Consider the polytopes $P_A$ and $P_B$ given by the systems of linear 
inequalities
\[
  \begin{pmatrix} 1\\-1\\0\end{pmatrix} 
  + \begin{bmatrix} -1&-1\\1&0\\0&1 \end{bmatrix}x \geq 0 \text{ and }
    \begin{pmatrix} 0\\2\\2 \end{pmatrix} 
  + \begin{bmatrix} 1&0\\-1&-1\\-1&1 \end{bmatrix}x \geq 0 ,
\]
respectively.
$P_A$ is the singleton $\{(1,0)\}$ and $P_B$ is a simplex containing $P_A$.
There is no matrix $C$ satisfying the conditions in part (2) of
Theorem~\ref{thm:pih-h} (with $m=0$). 
Indeed, $b=Ca$ implies $0=C_{11}-C_{12}$ and $B=CA$ implies 
$1=B_{11}=(-C_{11}+C_{12},-C_{11})$, a contradiction.
A solution to the linear feasibility system in Theorem~\ref{thm:pih-h} (1)  is
given by
\[
  c_0 = \mathds{1}_3,\ C = \begin{bmatrix} 0&1&0\\1&0&0\\1&0&2\end{bmatrix} .
\]
Moreover, it is easy to see that for any $P_B = \{x\in\R^2\ |\ b+Bx\geq 0\}$
containing $P_A$ containment is certified if and only if $B$ has the form
$B=[-b,-b+c]$ for some vector $c$.

  (2)
Consider the half space given by the linear polynomial $a(x)=1-x_1-x_2$. Let
$b(x)=b+[B_1,B_2]x$ be any half space. The condition in part (2) of the
Theorem~\ref{thm:pih-h} is satisfied if and only if $b=c,\ B_1=-c,\ B_2=-c$
for $c\geq 0$. Thus either $b(x)\equiv 0$ or $b(x)$ is a positive multiple of
$a(x)$.
\end{ex}

To prove Theorem~\ref{thm:pih-h} we use the following affine form of Farkas'
Lemma.

\begin{lemma}[{\cite[Corollary 7.1h]{Schrijver1986}}] \label{lem:farkas_aff}
Let $P=\{ x\in\R^d\ |\ a+Ax \geq 0\}$ be nonempty.
Then every affine polynomial $f\in\R[x]$ nonnegative on $P$ can be written as
$f(x) = c_0 + \sum_{i=1}^m c_i (a+Ax)_i$ with nonnegative coefficients $c_i$. 
\end{lemma}

\begin{proof}[Proof of Theorem~\ref{thm:pih-h}]
If $B=CA,\ 0=CA'$, and $b=Ca$ (resp. $b=c_0 + Ca$) with a nonnegative matrix
$C$, for any $x \in\pipa$ we have 
\[
  b+Bx+0y = C \left(a+Ax+A'y \right) \geq 0 ,
\]
i.e., $\pipa\subseteq P_B$.

Conversely, if $\pipa\subseteq P_B$, then any of the linear polynomials
$(b+Bx+0y)_i$, $i\in [l]$, is nonnegative on $P_A$. 
Hence, by Lemma~\ref{lem:farkas_aff}, $(b+Bx+0y)_i$ can be written as a linear
combination
\[
  (b+Bx+0y)_i\ =\ c'_{i0} + \sum_{j=1}^k c'_{ij} (a+Ax+A'y)_j
\]
with nonnegative coefficients $c'_{ij}$. Comparing coefficients yields
$ b_i = c'_{i0} + \sum_{j=1}^k c'_{ij} $ for $i\in [l]$, implying part (1) of
the statement.

To prove the second part, first translate both $P_A$ and $P_B$ to the origin. 
By assumption, there exists $(\bar{x},\bar{y})\in P_A$. 
Define $\bar{a}:=a+A\bar{x}+A'\bar{y}$ and $\bar{b}:=b+B\bar{x}$. 
Then $\bar{a}\geq 0$ and $0\in\{x\in\R^d\ |\ \bar{a}+Ax+A'y\geq 0\}$, implying
\[
  \bar{b}=C\bar{a},\ B=CA,\ 0 =CA' 
  \iff b=Ca,\ B=CA,\ 0 =CA' .
\]
Thus w.l.o.g. let $a\geq 0$.

Stiemke's Transposition Theorem~\cite[Section 7.8]{Schrijver1986} implies the
existence of a $\lambda>0$ such that $[A^T,A'^T]\lambda =0$, and thus
\[
  \lambda^T (a+Ax+A'y) = \lambda^T a = 1
\]
after an appropriate rescaling.
Note that $a\neq 0$ as otherwise $P_A =\{0\}$ is a singleton.
By multiplying that equation with $c'_{i0}$ from above, we obtain nonnegative 
$c''_{ij}$ with $\sum_{j=1}^k c''_{ij} (a+Ax+A'y)_j = c'_{i0}$, yielding 
\[
  (b+Bx)_i \ = \ \sum_{j=1}^k (c'_{ij}+c''_{ij}) (a+Ax+A'y)_j .
\]
Hence, $C = (c_{ij})_{i,j=1}^k$ with $c_{ij} := c'_{ij} + c''_{ij}$ is a
nonnegative matrix with $B = CA$, $0=CA'$, and 
$ (Ca)_i = \sum_{j=1}^k (c'_{ij} + c''_{ij}) a_j
  = b_i - c'_{i0} + c'_{i0}\,\lambda^T a = b_i $ for every $i\in[l]$.
\end{proof}

The sufficiency part of Theorem~\ref{thm:pih-h} can be extended to the case of
projected spectrahedra via the normal form~\eqref{eq:normalform} of a
(projected) polyhedron $P_A$ as a (projected) spectrahedron,
\[
  \pipa\ =\ \left\{x\in\R^d\ |\ \exists y\in\R^m :\  
  A(x,y) = \diag(a_1(x,y), \ldots, a_k(x,y)) \succeq 0 \right\},
\]
where $a_i(x,y)$ is the $i$th entry of the vector $a + Ax + A'y$. 

\begin{cor} \label{cor:pih-h}
Let $A(x,y)\in\sym^k[x,y]$ and $B(x)\in\sym^l[x]$ be normal forms of
polyhedra~\eqref{eq:normalform}. 
\begin{enumerate}
  \item
  $\pisa\subseteq S_B$ if and only if there exist positive
semidefinite diagonal matrices $C_0,\ C$ such that
\begin{equation} \label{eq:inclusion-pip}
  B_0 = C_0 + \sum_{i=1}^k (A_{0})_{ii} C_{ii},\ 
  B_p = \sum_{i=1}^k (A_{p})_{ii} C_{ii}\ \forall p\in[d],\ 
  0 = \sum_{i=1}^k (A'_{p})_{ii} C_{ii}\ \forall p\in[m].
\end{equation}
  \item
  Let $S_A$ be a polytope that is not a singleton. $\pisa\subseteq S_B$ if and 
  only if system~\eqref{eq:inclusion-pip} has a solution with $C_0=0$.
\end{enumerate}
\end{cor}

If the diagonality condition on the matrix $C$ in Corollary~\ref{cor:pih-h} is
dropped, then the above SDFP yields a sufficient condition for the
$\pis$-in-$\sym$ containment problem. 
Subsequently, the indeterminate matrix $C=\left(C_{ij}\right)_{i,j=1}^{k}$ is
a symmetric $kl\times kl$-matrix, where the $C_{ij}$ are $l\times l$-blocks.

\begin{thm} \label{thm:inclusion-proj}
Let $A(x,y)\in\sym^k[x,y]$ and $B(x)\in\sym^l[x]$ be linear pencils. 
Denote by $\pisa$ the coordinate projection of the spectrahedron $S_A$. 
If there exist positive semidefinite matrices
$C = (C_{ij})_{i,j=1}^k\in\psd^{kl}$ and $C_0\in\psd^l$ such that
\begin{align} \label{eq:inclusion-proj}
\begin{split}
  B_0 = C_0 + \sum_{i,j=1}^k (A_0)_{ij} C_{ij} ,\,
  B_p = \sum_{i,j=1}^k (A_p)_{ij} C_{ij} \ \forall p \in [d] ,\ 
  0 = \sum_{i,j=1}^k (A'_{p})_{ij} C_{ij} \ \forall p \in [m] ,
\end{split}
\end{align}
then $\pisa\subseteq S_B$.
\end{thm}

In the non-projected case, the sufficient semidefinite
criterion~\eqref{eq:inclusion-proj} has first been developed by Helton et
al.~\cite{Helton2010} using the theory of positive linear maps (cf.
Section~\ref{sec:pispispos}) and has been reproofed in~\cite{ktt1} by
elementary methods. In~\cite{kphd} the author showed that the condition is
exactly the 0th step of the hierarchy based on truncation of the Hol-Scherer
quadratic module~\eqref{eq:qm-matrix2}.
Here we are bringing this forward to the projected case.

For completeness we state a short proof of Theorem~\ref{thm:inclusion-proj}
based on~\cite{ktt1}.

\begin{proof}
We have
\begin{align*} 
\begin{split} 
  B(x) &= B_0 + \sum_{p=1}^d x_{p} B_{p} 
  =\ C_0 + \sum_{i,j=1}^{k} \left(A(x,y)\right)_{ij} C_{ij} 
  = C_0 + \mathbb{I}^T \left( (A(x,y))_{ij} C_{ij}\right)_{i,j=1}^{k}
\mathbb{I} 
\end{split}
\end{align*}
with $\mathbb{I} = \left[ I_l, \ldots, I_l \right]^T \in\R^{kl\times l}$.
Let $x\in\pisa$. By definition, there exists $y\in\R^m$ such that
$A(x,y)\succeq 0$. Thus the Kronecker product $A(x,y)\otimes C$ is positive
semidefinite.
Since $ \left((A(x,y))_{ij} C_{ij}\right)_{i,j=1}^k $ is a principal submatrix
of $A(x,y)\otimes C$, we have $B(x)\succeq 0$ as well.
\end{proof}

Even for the non-projected case, the sufficient semidefinite
criterion~\eqref{eq:inclusion-proj} is not necessary for containment in
general; see~\cite[Section 6.1]{ktt1}.

\subsection{A Sophisticated Positivstellensatz} \label{sec:sophisticated}

Consider the linear pencils $A(x,y)\in\sym^k[x]$ and $B(x)\in\sym^l[x]$. Then 
$\pisa$ is contained in $S_B$ if and only if $B(x) \succeq 0$ on $\pisa$. 
If $S_A$ is a spectratope, then this is equivalent to $B(x)+\varepsilon
I_l\in\qm^l(A)$ for all $\epsilon >0$, where
\[
  \qm^l(A) = \left\{ S_0 + \left\langle S,A(x,y)\right\rangle_l\ |\ 
  S_0\in\SOS^l[x,y],\ S\in\SOS^{kl}[x,y] \right\}
\]
is the quadratic module associated to $A(x,y)$ as defined
in~\eqref{eq:qm-matrix2}. Clearly, if $B(x)\in\qm^l(A)$ for a linear pencil
$B(x)\in\sym^l[x]$, then $\cl\pisa\subseteq S_B$.
Thus truncation of the $\qm^l(A)$ yields a hierarchy of SDFPs to decide
$\pis$-in-$\sym$ containment.

The drawback of this approach to the $\pis$-in-$\sym$ containment problem 
is that it relies on the geometry of the spectrahedron $S_A$ rather than its 
projection, namely the boundedness assumption on $S_A$ and the appearance of
the projection variables $y$ in the quadratic module.
In the following, we address this by developing a refinement of Hol-Scherer's
Positivstellensatz.
Particularly, we can eliminate the variables $y$ in the sense that they
neither appear in the quadratic module nor in the relaxation. 

Gouveia and Netzer~\cite{Gouveia2011} derived a Positivstellensatz for 
polynomials positive on the closure of a projected spectrahedron. 

\begin{prop}[{\cite[Theorem 5.1]{Gouveia2011}}] \label{prop:gouveia}
Let $A(x,y)\in\sym^k[x,y]$ be a strictly feasible linear pencil. Define the
quadratic module
\begin{align*}
  \qm(\pi A) =& \left\{ s_0 + \left\langle S,A(x,0) \right\rangle \ |\ 
  \left\langle S,A'_{i} \right\rangle = 0 \ \forall i\in[m],\ 
  s_0\in\SOS[x],\ S\in\SOS^k[x] \right\}.
\end{align*}
If $\pisa$ is bounded, then $\qm(\pi A)$ is Archimedean and contains all
polynomials positive on the closure of $\pisa$.
\end{prop}

Subsequently, we state and proof an extension to linear pencils positive 
definite on a projected spectrahedron. Thereto define the quadratic module
\begin{align} \label{eq:qm-matrix-proj}
  \qm^l(\pi A) = \left\{ S_0 + \left\langle S,A(x,0) \right\rangle_l\ |\ 
  \left\langle S,A'_{i} \right\rangle_l = 0\ \forall i\in [m],\ 
  S_0\in\SOS^l[x],\ S\in\SOS^{kl}[x] \right\} .
\end{align}
It is easy to see that $\qm^l(\pi A)$ is in fact a quadratic module.
Note that $\qm^l(\pi A)$ does not have to be finitely generated;
see~\cite[Section 5]{Gouveia2011}.
Clearly, every element of $\qm^l(\pi A)$ is positive semidefinite on the 
closure of $\pisa$.

\begin{thm} \label{thm:pispossatz}
Let $A(x,y)\in\sym^k[x,y]$ be a strictly feasible linear pencil such that
$\pisa$ is bounded.
For $l\in\N$ the quadratic module $\qm^l(\pi A)$ is Archimedean and contains
every matrix polynomial positive definite on $\cl\pisa$.
\end{thm}

\begin{proof}
By boundedness of $\pisa$ there exists $N\in\N$ sufficiently large such that
$N\pm x_i$ is nonnegative on $\pisa$ for all $i\in [d]$. 
We show that under the preconditions in the theorem $\qm^l(\pi A)$ contains 
every linear polynomial nonnegative on $\pisa$.
Then the quadratic module is Archimedean.

Let $b(x)=b_0+b^T x\in\R[x]_1$ be a fixed but arbitrary affine linear
polynomial nonnegative on $\pisa$. 
Consider the following primal-dual pair of SDPs. 

\vspace{-5pt}
\begin{minipage}{0.5\textwidth}
\begin{align*}
  p^* := \inf\ &\ b(x) \\
  \st\ &\ A(x,y) \succeq 0
\end{align*}
\end{minipage}
\begin{minipage}{0.45\textwidth}
\begin{align*}
  \sup\ &\ \left\langle -A_0,Z \right\rangle \\
  \st\ &\ \left\langle A_i,Z \right\rangle = b_i\ \forall i\in [d] \\
  \ &\ \left\langle A'_{i},Z \right\rangle = 0\ \forall i\in [m] \\
  \ &\ Z\in\psd^k
\end{align*}
\end{minipage}
\vspace{5pt}

\noindent Since $A(x,y)$ is strictly feasible by assumption, the dual problem 
(on the right-hand side) has optimal value $p^*-b_0$ and attains it;
see~\cite[Theorem 2.2]{deklerk2002}. 
Since $b(x)\geq 0$ on $\pisa$, we have $-b_0 \leq p^* -b_0$ and thus
\[
  b_0 -z_0 = \left\langle A_0,Z\right\rangle,\ \left\langle A_i,Z\right\rangle
 = b_i\ \forall i\in [d],\ \left\langle A'_{i},Z \right\rangle = 0\ \forall
  i\in [m]
\]
for some $Z\in\psd^k$ and $z_0 \geq 0$. Define $S(x)$ as the block diagonal 
$kl\times kl$-matrix with $l$ copies of $Z$ on its diagonal, i.e.,
$S(x) = \oplus_{j=1}^l Z$, and $S_0(x) = z_0 I_l$. Then
\begin{align*}
  S_0(x) + \left\langle S(x),A(x,0)\right\rangle_l 
  &= z_0 I_l + \bigoplus_{j=1}^l \left\langle Z,A(x,0) \right\rangle_l 
  = b(x) I_l 
\end{align*}
and 
$\langle S(x),A'_i\rangle_l = \oplus_{j=1}^l \langle Z,A'_i\rangle = 0 $ 
for $i\in [m]$. This implies $b(x)\in\qm^l(\pi A)$. 
By Hol-Scherer's Theorem, every matrix polynomial positive definite on
$\cl\pisa$ is contained in $\qm^l(\pi A)$.
\end{proof}

Theorem~\ref{thm:pispossatz} leads to a refined hierarchy for the
$\pis$-in-$\sym$ containment problem using the truncated quadratic module
\begin{equation}\label{eq:soshierarchy-proj}
  \qm_t^l(\pi A) = \left\{ S_0 + \left\langle S,A(x,0) \right\rangle_l\ |\ 
  \left\langle S,A'_{i} \right\rangle_l = 0\ \forall i\in [m],\ 
  S_0\in\SOS^l_t[x],\ S\in\SOS^{kl}_t[x] \right\} .
\end{equation}
It is evident from the definition of the quadratic modules $\qm^l(A)$ and
$\qm^l(\pi A)$ that the latter approach is preferable to the naive way from the
theoretical viewpoint (provided that $A(x,y)$ is strictly feasible).

\begin{cor}
Let $A(x,y)\in\sym^k[x,y]$ be a strictly feasible linear pencil such that
$\pisa$ is bounded and let $B(x)\in\sym^l[x]$ be a linear pencil.
\begin{enumerate}
  \item 
  $\pisa\subseteq S_B$ if and only if $B(x)+\varepsilon I_l\in\qm^l(\pi A)$ for
all $\varepsilon>0$.
  \item
  If $B(x)\succ 0$ on $\pisa$, then $B(x)\in\qm^l(\pi A)$.
\end{enumerate}
\end{cor}

The 0-th step of the hierarchy based on~\eqref{eq:soshierarchy-proj} is
exactly the sufficient containment criterion stated in
Theorem~\ref{thm:inclusion-proj}.
Indeed, for $t=0$, the constant sos-matrix $S$ equals the positive
semidefinite matrix $C$ after permuting rows and columns simultaneously.

\begin{prop} \label{prop:sdfp_sos_proj}
Let $A(x,y)\in\sym^k[x,y]$ and $B(x)\in\sym^l[x]$ be linear pencils. Assume 
$A(x,y)$ is strictly feasible. The following are equivalent.
\begin{enumerate}
  \item
  $B(x)\in\qm^l_0(\pi A)$.
  \item 
  There exist $C'\in\psd^{kl}$ and $C'_0\in\psd^l$ such that
  \begin{equation*}
    B_0 = C'_0 + \left\langle A_0,C'\right\rangle_l,\ 
  B_{p} = \left\langle A_p, C' \right\rangle_l\ \forall p\in [d],\
  0 = \left\langle A'_{q}, C' \right\rangle_l\ \forall q\in [m] .
  \end{equation*}
  \item There exist $C\in\psd^{kl},\ C_0\in\psd^l$ such that 
  \begin{equation*}
  B_0 = C_0 + \sum_{i,j=0}^k (A_0)_{ij} C_{ij},\ 
  B_p = \sum_{i,j=0}^k (A_p)_{ij} C_{ij}\ \forall p\in [d],\ 
  0 = \sum_{i,j=0}^k (A'_q)_{ij} C_{ij}\ \forall q\in [m] .
  \end{equation*}
\end{enumerate}
\end{prop}

\begin{proof}
The equivalence of (1) and (2) follows from the definition of the truncated
quadratic module by rewriting it as an SDFP.
Applying a simultaneous permutation of the rows and columns of $C'$ in (2)
(resp. of $C$ in (3)), the linear systems can easily be transformed.

For details in the non-projected case see~\cite[Theorem 5.1.11]{kphd}
\end{proof}

The proof of Theorem~\ref{thm:pispossatz} evidently yields necessity for the
$\pis$-in-$\HH$ containment problem. This, in particular, shows the
(theoretical) effectiveness of the approach based on
Theorem~\ref{thm:pispossatz}.

\begin{thm} \label{thm:pis-h}
Let $A(x,y)\in\sym^k[x,y]$ be a strictly feasible linear pencil and let the
coefficients of the linear pencil $B(x)\in\sym^l[x]$ be simultaneously 
congruent to a diagonal matrix.
\begin{enumerate}
  \item
  $\pisa\subseteq S_B$ if and only if $B(x)\in\qm^l_0(\pi A)$.
  \item
  Assume $S_B$ is a polytope with nonempty interior. Then
  $\pisa\subseteq S_B$ if and only if $B(x)\in\qm^l_0(\pi A)$ with $S_0=0$.
\end{enumerate}
In particular, the statements (1) and (2) hold for a diagonal linear pencil
$B(x)$, i.e., a polyhedron in normal form~\eqref{eq:normalform}.
\end{thm}

In order to prove Theorem~\ref{thm:pis-h}, we use natural adaptions of
auxiliary results on the behavior of the sufficient containment criterion with
regard to block diagonalization and transitivity as shown in~\cite{ktt1} to
the non-projected setting.
Using Proposition~\ref{prop:sdfp_sos_proj}, it is easy to verify the validity
of these statements.

\begin{proof}
As for $t=0$ the resulting SDFP is invariant under non-singular congruence
transformations of $B(x)$ (see~\cite[Lemma 5.1.14]{kphd}), we can retreat to
the normal form~\eqref{eq:normalform} 
$B(x) = \bigoplus_{q=1}^l b^q(x)\in\sym^l[x]$ with $b^q(x) = b_0^q +x^T b^q$
for $q\in[l]$.
Denote by $b_0^q,b_1^q,\ldots,b_d^q$ the coefficients of the linear form 
$ b^q(x) = (b_0+Bx)_q $. Set $ b^q := (b_1^q, \ldots, b_d^q) $.

The proof of Theorem~\ref{thm:pispossatz} yields certificates
\[
  b_0^q -z_0^q = \left\langle A_0,Z^q \right\rangle,\ 
  \left\langle A_i,Z^q \right\rangle = b_i^q\ \forall i\in [d],\ 
  \left\langle A'_{i},Z^q \right\rangle = 0\ \forall i\in [m]
\]
for some $Z^q \in\psd^k$ and $z_0^q \geq 0$.
Setting $S(x)=\bigoplus_{q=1}^l Z^q$ and $S_0(x)=\bigoplus_{q=1}^l z_0^q$, 
this implies part (1) of the statement. 

To prove the second part, let $S(x)$ as before and set $S_0(x)$ to be zero. 
Then 
\begin{align*}
  \left\langle S(x),A(x,0)\right\rangle_l 
  = \bigoplus_{q=1}^l \left\langle A(x,0),Z^q \right\rangle_l 
  = \bigoplus_{q=1}^l \left( f_0-z_0^q + \sum_{i=1}^d b_i^q x_i \right)
\end{align*}
certifies the containment $\pisa\subseteq S_{B'}$, where $B'(x)$ is defined as
\[
  B'(x) = \bigoplus_{q=1}^{l} \left(r^q+\sum_{p=1}^d x_p \right) .
\]
Assuming that $S_B$ is a polytope, we have $S_{B'}\subseteq S_B$ and thus, by
transitivity and exactness of the initial hierarchy step for
polytopes, see Corollary~\ref{cor:pih-h}, there is a certificate for the
containment question $\pisa\subseteq S_B$ of degree zero with $S_0(x)=0$.
\end{proof}

As a special case of Theorem~\ref{thm:pispossatz}, we gain a
Positivstellensatz for polynomials on projected polyhedra having boundedness
as its only precondition.

\begin{prop}\label{prop:pihpossatz}
Let $P_A = \{(x,y)\in\R^{d+m}\ |\ a+Ax+A'y \ge 0 \}$ be a nonempty
polyhedron such that $\pipa$ is bounded. The quadratic module 
\[
  \qm^1(\pi,A) = \left\{ s_0 + \sum_{i=1}^k s_i(x) (a+Ax)_i\ |\ 
  \sum_{i=1}^k s_i(x) (A'_{i,j})= 0\ \forall j\in [m],\ 
  s_0,\ldots,s_k\in\SOS[x] \right\}
\]
is Archimedean and contains every polynomial positive on $\pipa$.
\end{prop}

\begin{proof}
The proof follows from the proof of Theorem~\ref{thm:pispossatz} by retreating
to diagonal pencils and the fact that strong duality holds for linear
programming~\cite[Corollary 7.1g]{Schrijver1986}.
\end{proof}

\subsection{Examples}\label{sec:examples}

We discuss some academic examples for the hierarchy stated
in~\eqref{eq:soshierarchy-proj}.
All computations are made on a desktop computer with Intel Core i3-2100 @ 3.10
GHz and 4 GB of RAM.
In the tables, ``time'' states the time in seconds for setting up the
problem in YALMIP~\cite{YALMIP} and solving it with Mosek~\cite{andersen2000}.

In the examples we consider the optimization version
of~\eqref{eq:soshierarchy-proj}
\begin{align*}
  \mu(t)\ =\ \sup\ &\ \mu \\ \st\ &\ B(x)-\mu I_l \in\qm^l_t(\pi A) .
\end{align*}
Letting $t$ tend to infinity, the sequence of optimal values $\mu(t)$
converges to the value 
$\mu^* = \sup\{ \mu\ |\ B(x)-\mu I_l \succeq 0\ \forall x\in \pisa\}$ 
which is nonnegative if and only if $\pisa\subseteq S_B$. Thus a nonnegative
value $\mu(0)$ states the existence of a containment certificate with $t=0$.

A $d$-dimensional \emph{ball} is a spectratope $S_A$ given by the linear pencil
\begin{equation} \label{eq:ellipsoid}
  A(x) = I_{d+1} + \sum_{p=1}^{d} \frac{x_p}{r}(E_{p,d+1}+E_{d+1,p})
  \in\sym^{d+1}[x]
\end{equation}
with $r>0$.

\begin{table}
 \begin{tabular}{lll|rl}
  \toprule
  $\pisa$ & $rS_B$ & $r$ & time & $\ \mu(0)$\\
  \midrule
  two disks & 2-ball
    & 1.99 & 0.7978 & $ -0.0050 $ \\
  & & 2    & 0.8215 & \hspace*{8pt}$ 5.9978\cdot 10^{-08} $ \\
  & & 2.01 & 0.9173 & \hspace*{5pt}$ 0.0050$ \\ 
  \toprule
  $ S_A $ & $rS_B$ &     &  &   \\
  \midrule
   & 3-ball
    & 2.23 & 0.8690 & $ -0.0027 $ \\
  & & 2.2361 & 0.7470 & \hspace*{9pt}$ 1.4339\cdot 10^{-05}$ \\
  & & 2.24 & 0.8803 &\hspace*{5pt} $ 0.0018 $ \\
  \bottomrule
 \end{tabular}
\\[+1ex]
\caption{Computational test of containment as described in
Example~\ref{ex:twodisks}.}
\label{tab:twodisks}
\end{table}

\begin{figure}[ht]
\centering
  \includegraphics[width=0.25\textwidth]{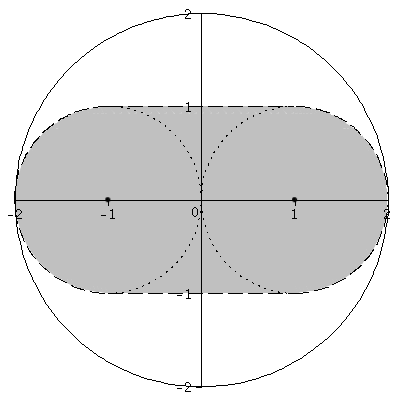}
  \caption{The convex hull of two disks in a $2$-ball as stated in
Example~\ref{ex:twodisks}.}
\label{fig:twodisks-ball}
\end{figure}

\begin{ex}\label{ex:twodisks}
Consider the convex hull of two disks $\pisa$ as defined in
Example~\ref{ex:pispis} and the 2-ball of radius $r>0$.
It follows from the construction of $\pisa$ that it is centrally symmetric and
that its circumradius is 2; see Figure~\ref{fig:twodisks-ball}.
Up to numerical accuracy, this value is computed by our approach; see
Table~\ref{tab:twodisks}.
\end{ex}

\begin{ex}\label{ex:tvscreen}
Consider the TV screen $\pisa$ as defined in Example~\ref{ex:pispis}.
Note that while the TV screen is centrally symmetric (as its boundary equals 
the variety defined by the polynomial $1-x_1^4-x_2^4$), its defining 
spectrahedron is not (as the point $(1,0,1,0)$ is contained in $S_A$ but its 
negative $(-1,0,-1,0)$ is not).

As one can see in Table~\ref{tab:tvscreen}, the circumradius of the TV screen
is at most $\sqrt[4]{2}\approx 1.1892$, while the ``centrally symmetric
circumradius'' of $S_A$ is at most $\sqrt{\sqrt{2}+1}\approx 1.5538$.

Actually, the computed values for the (centrally symmetric) circumradius are
exact. For 
$ p = \left(\frac{1}{\sqrt[4]{2}},\frac{1}{\sqrt[4]{2}},
  \frac{1}{\sqrt{2}},\frac{1}{\sqrt{2 }}\right) \in S_A$ 
we have $\|\pi(p)\|_2 = \sqrt[4]{2}$ and $\|p\|_2=\sqrt{\sqrt{2}+1}$, implying
that the circumradius of the TV screen is at least $\sqrt[4]{2}$ and that
$\sqrt{\sqrt{2}+1}$ is the smallest possible radius of a ball (centered at the
origin) containing $S_A$.
\end{ex}

\begin{table}
 \begin{tabular}{lll|rl}
  \toprule
  $\pisa$ & $rS_B$ & $r$ & time & $\ \mu(0)$\\
  \midrule
  TV screen & 2-ball & 1.18 & 0.8514 & $ -0.0078 $ \\
  & & $\sqrt[4]{2}$ & 2.0564 & $ -1.3621\cdot 10^{-08} $ \\
  & & 1.19          & 0.9964 & \hspace*{9pt}$ 6.6628\cdot 10^{-04} $ \\ 
  & & 1.2           & 0.9854 & \hspace*{5pt} $ 0.0090 $ \\ 
  \toprule
  $ S_A $ & $rS_B$ &         &  &   \\
  \midrule
   & 4-ball & 1.55 & 1.1036 & $ -0.0024 $ \\
  & & $\sqrt{\sqrt{2}+1}$ & 0.9373 & \hspace*{9pt}$3.1037\cdot 10^{-09}$ \\
  & & 1.56 & 1.0723 &\hspace*{5pt} $ 0.0040 $ \\
  \bottomrule
 \end{tabular}
\\[+1ex]
\caption{Computational test of containment as described in
Example~\ref{ex:tvscreen}.}
\label{tab:tvscreen}
\end{table}
%

\subsection{Containment of Projected Spectrahedra and Positive Linear Maps}
\label{sec:pispispos}

We discuss an extension of the connection between positive linear maps and
containment of spectrahedra (as introduced by Helton et al.~\cite{Helton2010};
see also~\cite{ktt2}) to projected spectrahedra.

Given two linear pencils $A(x,y)\in\sym^k[x,y]$ and $B(x)\in\sym^l[x]$ with
$y=(y_1,\ldots,y_m)$ define the linear subspaces
\begin{align*}
  \Acal = \linspan\{ A_0,\ldots,A_d,A'_{1},\ldots,A'_{m} \}\ \text{and}\ 
  \Bcal = \linspan\{ B_0,\ldots,B_d \}. 
\end{align*}
Every element in $\Acal$ can be associated to a homogeneous linear pencil
$ A(x_0,x,y) \in\sym^k[x_0,x,y] $ ($A_0$ being the coefficient of $x_0$). 
The linear pencil
\begin{align*}
  \widea(x_0,x,y) 
  &:= x_0 (1\oplus A_0) + \sum_{p=1}^{d} x_p (0 \oplus A_p) 
    + \sum_{q=1}^{m} y_q (0 \oplus A'_{q})
\end{align*}
is called the \emph{extended linear pencil} associated to $A(x_0,x,y)$. The
associated linear subspace is 
$ \widema = \linspan\{ 1\oplus A_0,0 \oplus A_1,\ldots,0 \oplus A_d,0 \oplus
  A'_1,\ldots,0 \oplus A'_{m} \} $.

For linearly independent $A_1,\ldots,A_d,A'_{1},\ldots,A'_{m}$, let 
$ \widephi_{AB}:\ \widema\rightarrow\Bcal $ be the linear map defined by
\[
  \widephi_{AB} (1 \oplus A_0) = B_0,\ 
  \widephi_{AB} (0 \oplus A_p) = B_p \ \forall p \in [d] ,\ 
  \widephi_{AB} (0 \oplus A'_{p}) = 0 \ \forall p \in [m] .
\]
Since every linear combination 
$ 0 = \lambda_0 (1\oplus A_0) + \sum_{p=1}^{d} \lambda_p (0\oplus A_p) + 
\sum_{q=1}^{m} \lambda_{d+q} (0\oplus A'_q)$ for
real scalars $\lambda_0,\ldots,\lambda_{d+m}$ yields $\lambda_0 = 0$, it
suffices to assume the linear independence of the coefficient matrices
$A_1,\ldots,A_d,A'_{1},\ldots,A'_{m}$ to ensure that $\widephi_{AB}$ is 
well-defined. 
If $A_0,A_1,\ldots,A_{d},A'_1,\ldots,A'_m$ are linearly independent, then we
can retreat to the simpler map $\Phi_{AB}:\ \Acal\rightarrow\Bcal$ defined by 
\begin{align*}
  &\Phi_{AB}(A_p) = B_p \ \forall p \in [d]\text{ and }
  \Phi_{AB}(A'_{q}) = 0 \ \forall q \in [m] .
\end{align*}

The next theorem extends the key connection between operator theory and
containment of spectrahedra to the setting of projections of spectrahedra.

\begin{thm} \label{thm:pispos}
Let $A(x,y)\in\sym^k[x,y]$ and $B(x)\in\sym^l[x]$ be linear pencils. 
\begin{enumerate}
  \item 
  If $\Phi_{AB}$ or $\widephi_{AB}$ is positive, then $\pisa\subseteq S_B$.
  \item
  If $\pisa\neq\emptyset$, then $\pisa\subseteq S_B$ implies positivity of
  $\widephi$.
  \item
  If $\pisa\neq\emptyset$ and $S_A$ is bounded, then $\pisa\subseteq S_B$
  implies positivity of $\Phi$.
\end{enumerate}
\end{thm}

\begin{proof}
\ 

(1):
Let $\Phi_{AB}$ be positive. For every $x\in\pisa$ there exists $y\in\R^m$
such that $A(x,y)\succeq 0$, i.e., $A(x,y)\in\psd^k\cap\Acal$.
Then 
$ B(x) = B(x)+\sum_{q=1}^{m} y_q\Phi(A'_q) = \Phi(A(x,y)) \in\psd^l\cap\Bcal$ 
and hence $x\in S_B$.
There is no difference in the proof if $\widephi_{AB}$ is positive.


(2):
Since the spectrahedra defined by $A(x,y)$ and $\widea(x,y)$ coincide, their
projections equal and hence $\pi\left(S_{\widea}\right)\subseteq S_B$. 
Let $\widea(x_0,x,y) \in\psd^{k+1}\cap\widema$. Then $x_0\geq 0$.

Case $x_0 > 0$. 
By scaling the linear pencil with $1/x_0$ the positive semidefiniteness is
preserved.
Thus, $1/x_0 \widea(x_0,x,y) = \widea(1,x/x_0,y/x_0)\in\psd^{k+1}\cap\widema$
and $ x/x_0 \in\pisa\subseteq S_B $. 
Scaling $B(x/x_0)$ by $x_0$ yields
$ \widephi(\widea(x_0,x,y)) = x_0 B_0 + \sum_{p=1}^d x_p B_p 
  = x_0 B(x/x_0) \in\psd^l\cap\Bcal $.

Case $x_0 = 0$.
If $(x,x_0) =(0,0)$, the statement is obvious. Let $x\neq 0$.
Fix a point $\bar{x}\in\pisa\neq\emptyset$. Then, for some $\bar{y},y\in\R^m$, 
$ \widea(1,\bar{x}+tx,\bar{y}+ty) 
  = \widea(1,\bar{x},\bar{y}) + \widea(0,tx,ty) \succeq 0 $ 
for all $t>0$, implying $\bar{x}+tx \in\pisa\subseteq S_B$ for all $t>0$.
Thus $x$ lies in the recession cone of $\pisa$ which clearly is contained in
the recession cone of $S_B$. 
Indeed, $\frac{1}{t} B(1,\bar{x})+B(0,x) = \frac{1}{t} B(\bar{x}+tx)\succeq 0$
for all $t>0$. 
By closedness of the cone of positive semidefinite matrices, we get
$B(0,x)\succeq 0$. 
Hence, $\widephi(\widea(x_0,x,y)) =\widephi(\widea(0,x,y)) = B(0,x)\succeq 0$.


(3):
Let $A(x_0,x,y) = x_0 A_0 +\sum_{p=1}^d x_p A_p +\sum_{q=1}^m y_q
  A'_{q}\in \psd^k\cap\Acal$.

Case $x_0 > 0$. 
This case follows by a similar scaling argument as in part (2).

Case $x_0 \leq 0$. 
Since $\pisa\neq\emptyset$, there exists $\bar{x}\in\pisa$ and hence, for some
$\bar{y}\in\R^m$,
\begin{align*}
  A(0,x+|x_0|\bar{x},y+|x_0|\bar{y}) 
  \succeq |x_0| \cdot A(1,\bar{x},\bar{y}) \succeq 0 .
\end{align*}
For $A(0,x+|x_0|\bar{x},y+|x_0|\bar{y})\neq 0$, one has an improving ray of
the spectrahedron $S_A$, in contradiction to boundedness of $S_A$.
For $A(0,x+|x_0|\bar{x},y+|x_0|\bar{y}) = 0$, linear independence of 
$A_0,\ldots,A_{d+m}$ implies $(x+|x_0|\bar{x},y+|x_0|\bar{y}) = (0,0)$. But
then $x_0 A(1,\bar{x},\bar{y}) = A(x_0,x,y)\succeq 0$ together with 
$x_0\leq 0$ and $A(1,\bar{x},\bar{y})\succeq 0$ imply either
$A(1,\bar{x},\bar{y}) = 0$, in contradiction to linear independence, or
$(x_0,x) = (0,0)$. Clearly, in this case, $\Phi_{AB}(0) = 0$. 
\end{proof}

The linear map $\widephi_{AB}$ can be represented by an symmetric 
$(k+1)l\times (k+1)l$ matrix $\widehat{C}$.
By expecting the linear equations defining $\widehat{C}$, it is easy to see
that $\widehat{C}=C_0\oplus C$, where the matrix pair $C_0,\ C$ is from
Theorem~\ref{thm:inclusion-proj}.
Thus Proposition~\ref{prop:sdfp_sos_proj} implies the next corollary to
Theorem~\ref{thm:pispos}.

\begin{cor}
Let $A(x,y)\in\sym^k[x,y]$ and $B(x)\in\sym^l[x]$ be linear pencils with
$A(x,y)$ strictly feasible. Then the following are equivalent.
\begin{enumerate}
  \item
  The map $\widephi_{AB}$ is completely positive, i.e., $\widehat{C}\succeq 0$.
  \item
  The solitary criterion~\eqref{eq:inclusion-proj} is feasible.
  \item
  $B(x)\in\qm^l_0(\pi A)$.
\end{enumerate}
\end{cor}


\bibliography{containment-proj}
\bibliographystyle{plain}


\end{document}